\documentclass[12pt,a4paper]{amsart}
\usepackage{graphicx,latexsym,amsfonts,amsmath,amssymb,rotating,txfonts,mathrsfs,enumerate}
\usepackage{epic}
\usepackage{curves}
\usepackage{pdfsync}
\input xy
\xyoption{all}

{\bf}{\rm}
\newtheorem{theorem}{Theorem}[section]
\newtheorem{proposition}[theorem]{Proposition}
\newtheorem{corollary}[theorem]{Corollary}
\newtheorem{lemma}[theorem]{Lemma}
\newtheorem{definition}[theorem]{Definition}
\newtheorem{remark}[theorem]{Remark}

\newtheorem{example}[theorem]{Example}
{\bf}{\it}

\newcommand{\RR}{{\mathbb R }}
\newcommand{\CC}{{\mathbb C }}

\newcommand{\ZZ}{{\mathbb Z }}
\newcommand{\PP}{ {\mathbb P }}

\newcommand{\GG}{{\mathbb G }}

\newcommand{\calx}{\mathcal{X}}
\newcommand{\calo}{\mathcal{O}}

\newcommand{\cals}{\mathcal{S}}
\newcommand{\calv}{\mathcal{V}}

\newcommand{\reg}{\mathrm{reg}}

\newcommand{\deff}{\mathbf{D}}

\newcommand{\bbb}{\mathbf{b}}

\newcommand{\res}{\operatornamewithlimits{Res}}

\newcommand{\calox}[2]{\calo_{X_{#1}}(#2)}
\newcommand{\ires}{\res_{z_1=\infty}\res_{z_{2}=\infty}\dots\res_{z_k=\infty}}
\newcommand{\sires}{\res_{\mathbf{z}=\infty}}
\newcommand{\dbz}{\,d\mathbf{z}}

\newcommand{\coeff}{\mathrm{coeff}}

\newcommand{\grass}{\mathrm{Grass}}

\newcommand{\cotx}{T_X^*}

\newcommand{\vt}{\tilde{V}}
\newcommand{\xt}{\tilde{X}}

\newcommand{\bz}{\mathbf{z}}
\newcommand{\bi}{\mathbf{i}}
\newcommand{\baa}{\mathbf{a}}
\newcommand{\bj}{\mathbf{j}}

\newcommand{\kt}{{K}}

\newcommand{\bs}{\mathbf{s}}

\newcommand{\GL}{\mathrm{GL}}

\def\a{\alpha}

\def\d{\delta}

\def\l{\lambda}

\def\s{\sigma}

\title{Towards the Green-Griffiths-Lang conjecture via equivariant localisation} 

\author{Gergely B\'erczi}
\address{Mathematical Institute \\ University of Oxford \\ Andrew Wiles Building \\ OX2 6GG Oxford, UK}\email{berczi@maths.ox.ac.uk}
\thanks{This work was partially supported by the Engineering and Physical Sciences 
Research Council [grant numbers   GR/T016170/1,EP/G000174/1].}

\date{}
\begin{document}

\maketitle

\begin{abstract}
Green and Griffiths \cite{gg} and Lang \cite{lang} conjectured that for every complex projective algebraic variety $X$ of general type there exists a proper algebraic subvariety of $X$ containing all nonconstant entire holomorphic curves $f:\CC \to X$. Using equivariant localisation on the Demailly-Semple jet differentials bundle we give an affirmative answer to this conjecture for generic projective hypersurfaces $X \subset \PP^{n+1}$ of degree $\deg(X) \ge n^{9n}$.   
\end{abstract}

\section{Introduction}\label{sec:intro}

A central object in the study of polynomial differential equations on a smooth complex manifold is the bundle $J_kX$ of $k$-jets $(f',f'',\ldots, f^{(k)})$ of germs of holomorphic curves $f:\CC \to X$ over $X$ and the associated Green-Griffiths bundle  $E_{k,m}^{GG}=\calo(J_kX)$ of algebraic differential operators \cite{gg} whose elements are polynomial functions $Q(f',\ldots,f^{(k)})$ of weighted degree $m$.  In \cite{dem} Demailly introduced the subbundle $E_{k,m} \subset E_{k,m}^{GG}$ of jet differentials that are invariant under reparametrization of the source $\CC$. The group $\GG_k$ of $k$-jets of reparametrisation germs $(\CC,0) \to (\CC,0)$ at the origin acts fibrewise on $J_kX$ and $\oplus_{m=1}^\infty E_{k,m}=\calo(J_kX)^{U_k}$ is the graded algebra of invariant jet differentials under the maximal unipotent subgroup $U_k$ of $\GG_k$. This bundle gives a better reflection of the geometry of entire curves, since
it only takes care of the image of such curves and not of the way they are
parametrized. However, it also comes with a technical difficulty, namely, the reparametrisation group $\GG_k$ is non-reductive, and the classical geometric invariant theory of Mumford \cite{git} is not applicable to describe the invariants and the quotient $J_kX/\GG_k$; for details see \cite{bk,dk}. 

In \cite{dem} Demailly describes a  smooth compactification of $J_kX/\GG_k$ as a tower of projectivised bundles on $X$---the Demailly-Semple bundle---endowed with tautological line bundles $\tau_1,\ldots \tau_k$ whose sections are $\GG_k$-invariants. In \cite{dmr} the algebraic Morse inequalities of Trapani and Demailly reduce the existence of global invariant jet differentials on $X$ to the positivity of a certain intersection number on the Demailly-Semple tower. 

This paper introduces a new technique to handle the complexity and difficulties of computations with the cohomology ring of the Demailly-Semple tower in \cite{dmr}. We apply equivariant localisation in stages on the tower and transform the fixed point formula into an iterated residue to express intersection numbers of the tautological bundles on the Demailly-Semple bundle $\pi:X_k \to X$ as coefficients of the Laurent expansion of a rational function. The crucial idea of introducing iterated residues was motivated by the author's earlier work \cite{bsz}.
\begin{theorem}\label{maintechnical} Let $X \subset \PP^{n+1}$ be a smooth projective hypersurface and let $u_i=c_1(\tau_i)$ and $h=\pi^*c_1(\calo_{X}(1))$ denote the first Chern classes of the $i$th tautological line bundle on the Demailly-Semple tower $X_k$ and the tautological line bundle on $X$, respectively. For any homogeneous polynomial $P=P(u_1,\ldots, u_k,h)$ of degree $\deg(P)=\dim X_k=n+k(n-1)$ we have
\[\int_{X_{k}}P=\int_X \sires \frac{\prod_{2\le t_1 \le t_2 \le k} -(z_{t_1}+z_{t_1+1}+ \ldots +z_{t_2})P(z_1,\ldots, z_k,h)\dbz}{\prod_{1 \le s_1 < s_2 \le k} (z_{s_1}-z_{s_1+1}-\ldots -z_{s_2})\prod_{j=1}^k(z_1+\ldots +z_j)^n}\prod_{j=1}^k s\left(\frac{1}{z_1+\ldots +z_j}\right)
\]
where
\[s\left(\frac{1}{z_1+\ldots +z_j}\right)=1+\frac{s_1(X)}{z_1+\ldots +z_j}+\frac{s_2(X)}{(z_1+\ldots +z_j)^2}+\ldots +\frac{s_n(X)}{(z_1+\ldots +z_j)^n}\]
is the total Segre class at $1/(z_1+\ldots +z_j)$ and the iterated residue is equal to the coefficient of $(z_1\ldots z_k)^{-1}$ in the expansion of the rational expression in the domain $z_1\ll \ldots \ll z_k$. 
\end{theorem}

Note that the iterated residue on the right hand side of this formula is a degree $n$ cohomology class expressed as a polynomial in $h, s_1(X),\ldots, s_n(X)$ so the tautological integral over the Demailly tower is a polynomial purely in topological invariants of $X$. 

The Green-Griffiths-Lang (GGL) conjecture \cite{gg,lang} states that every projective algebraic variety $X$ of general type contains a proper algebraic subvariety $Y\subsetneqq X$ such that every
nonconstant entire holomorphic curve $f:\CC \to X$ satisfies $f(\CC) \subset Y$.  The GGL conjecture is related to the stronger concept of a hyperbolic variety \cite{kob}. A projective variety $X$ is hyperbolic (in the sense of Brody) if there is no nonconstant entire holomorphic curve in $X$, i.e. any holomorphic map $f: \CC \to X$ must be constant.   Hyperbolic algebraic varieties have attracted considerable attention, in part because of their conjectured diophantine properties. For instance, Lang \cite{lang} has conjectured that any hyperbolic complex projective variety over a number field K can contain only finitely many rational points over K.
 
A positive answer to the GGL conjecture has been given for surfaces by McQuillan \cite{mcquillan} under the assumption that the second Segre number $c^2_1-c_2$ is positive. Siu in \cite{siu1,siu2,siu3,siu4} developed a strategy to establish algebraic degeneracy of entire holomorphic curves in generic hypersurfaces $X\subset \PP^{n+1}$ of high degree, and also hyperbolicity of such hypersurfaces for even higher degree. Following this strategy combined with techniques of Demailly \cite{dem} the first effective lower bound for the degree of the hypersurface in the GGL conjecture was given by Diverio, Merker and Rousseau in \cite{dmr}. They proved that for a generic projective hypersurface $X\subset \PP^{n+1}$ of degree $\deg(X)>2^{n^5}$ the GGL conjecture holds. 

Proving algebraic degeneracy of holomorphic curves on $X$ means finding a nonzero polynomial function $P$ on $X$ such that all entire curves $f:\CC \to X$ satisfy $P(f(\CC))=0$. All known methods of proof are based on establishing first the existence of certain algebraic differential equations $P(f,f',\ldots,f^{(k)})=0$ of some order $k$, and then the second step is to find enough such equations so that they cut out a proper algebraic locus $Y\subsetneqq X$. 

Demailly in \cite{dem3} formulated a generalised version of the GGL conjecture for directed manifolds $(X,V)$, where $V \subseteq T_X$ is a subbundle, and proved--using holomorphic Morse inequalities and probabilistic methods--that for any projective directed manifold $(X,V)$ with $K_V$ big there is a differential equation $P$ of order $k\gg 1$ such that any entire curve $f$ must satisfy $P(f,f',\ldots, f^{(k)})=0$. Merker \cite{merker3} proved the same for projective hypersurfaces in $\PP^{n+1}$ of degree at least $n+3$ using algebraic Morse inequalities. Darondeau \cite{darondeau} adapted techniques of the present paper to study algebraic degeneracy of entire curves in complements of smooth projective hypersurfaces. Demailly \cite{dem15} proved the GGL conjecture for directed pairs $(X,V)$ satisfying certain jet stability conditions and announced the proof of the Kobayashi conjecture on the hyperbolicity of very general algebraic hypersurfaces and complete intersections. Siu \cite{siu4} proved the Kobayashi hyperbolicity of projective hypersurfaces of sufficiently high (but not effective) degree.   

This paper focuses on smooth projective hypersurfaces $X\subset \PP^{n+1}$. The main technical reason for this is that in this case the Chern classes of $X$ on the right hand side of Theorem \ref{maintechnical} are expressible with the degree $d$ of $X$ and the first Chern class $h$ of the hyperplane line bundle over $X$ using the identity
\[(1+h)^{n+2}=(1+dh)c(X),\]
where $c(X)=c(T_X)$ is the total Chern class of $X$. Then the iterated residue becomes a polynomial in $h^n$ with polynomial coefficients in $d,n$ and integration simply means the substitution $h^n=d$. 

This paper follows the strategy of \cite{dmr}, but the efficiency of computations with iterated residues allows us to prove the GGL conjecture with a sharper exponential bound on the degree of a the generic hypersurface.  We use the residue formula in Theorem \ref{maintechnical} to prove the existence of global differential equations of order $k=n$ satisfied by entire holomorphic curves on $X$ with $\deg(X)>6n^{8n}$. Combined with deformation arguments of \cite{dmr} (based on earlier works \cite{voisin,siu2,merker2}) to get enough independent differential equations, this gives us the following effective degree bound in the Green-Griffiths-Lang conjecture:
\begin{theorem}\label{mainthmone}
Let $X\subset \PP^{n+1}$ be a generic smooth projective hypersurface
of degree $\deg(X)\ge n^{9n}$. 
Then there exists a proper algebraic subvariety $Y\subsetneqq X$ such
that every nonconstant entire holomorphic curve $f:\CC \to X$ has
image contained in $Y$.
\end{theorem}
Even if our lower bound is far from the one $\deg(X)\ge n+3$ insuring general type, to our knowledge this is the best effective bound valid for generic projective hypersurfaces.  

In the forthcoming  paper \cite{b2} we replace the Demailly-Semple bundle with a more sophisticated  compactification of $J_kX/\GG_k$ motivated by the author's earlier work in global singularity theory \cite{bsz} on Thom polynomials of singularity classes. We will prove that the GGL conjecture for hypersurfaces with polynomial degree follows from a conjectural positivity property of Thom polynomials. 

\textsl{Acknowledgments} I would like to thank Damiano Testa and Frances Kirwan for patiently listening to details of this work. The first version of this paper was presented in Strasbourg, Orsay and Luminy in 2010/2011. I would like to thank to Jean-Pierre Demailly, Jo\"el Merker, Simone Diverio, Erwan Rousseau and Lionel Darondeau for their comments and suggestions. The paper has been rewritten based on these discussions to make the technical details of localisation more available to non-experts. 
The author warmly thanks Andr\'as Szenes, his former PhD supervisor, for the collaboration on \cite{bsz}, from which this paper has outgrown.

\section{Jet differentials}\label{sec:jetdiff}

The central object of this paper is the algebra of invariant jet differentials under
reparametrisation of the source space $\CC$. For more details see the survey papers \cite{dem,dr}.

\subsection{Invariant jet differentials}\label{subsec:jetdiff}
Let $X$ be a complex $n$-dimensional manifold and let $k$ be a positive integer. Green and Griffiths
in \cite{gg} introduced the bundle $J_kX \to X$
of $k$-jets of germs of parametrized curves in $X$; its
fibre over $x\in X$ is the set of equivalence classes of germs of holomorphic
maps $f:(\CC,0) \to (X,x)$, with the equivalence relation $f\sim g$
if and only if the derivatives $f^{(j)}(0)=g^{(j)}(0)$ are equal for
$0\le j \le k$. If we choose local holomorphic coordinates
$(z_1,\ldots, z_n)$ on an open neighbourhood $\Omega \subset X$
around $x$, the elements of the fibre $J_kX_x$ are represented by the Taylor expansions 
\[f(t)=x+tf'(0)+\frac{t^2}{2!}f''(0)+\ldots +\frac{t^k}{k!}f^{(k)}(0)+O(t^{k+1}) \]
 up to order $k$ at $t=0$ of $\CC^n$-valued maps $f=(f_1,f_2,\ldots, f_n)$
on open neighbourhoods of 0 in $\CC$. Locally in these coordinates the fibre  can be written as
\[J_kX_x=\left\{(f'(0),\ldots, f^{(k)}(0)/k!)\right\}=(\CC^n)^k,\]
which we identify with $\CC^{nk}$.  Note that $J_kX$ is not
a vector bundle over $X$ since the transition functions are polynomial but
not linear, see \cite{dem} for details.

Let $\GG_k$ denote the group of $k$-jets of local reparametrisations
of $(\CC,0) \to (\CC,0)$
\[t \mapsto \varphi(t)=\a_1t+\a_2t^2+\ldots +\a_kt^k,\ \ \ \a_1\in
\CC^*,\a_2,\ldots,\a_k \in \CC,\] 
under composition modulo terms $t^j$ for $j>k$. This group acts fibrewise on
$J_kX$ by substitution. A short computation shows that this is a
linear action on the fibre:
\begin{multline}\nonumber f \circ
\varphi(t)=f'(0)\cdot(\a_1t+\a_2t^2+\ldots
+\a_kt^k)+\frac{f''(0)}{2!}\cdot (\a_1t+\a_2t^2+\ldots
+\a_kt^k)^2+\ldots \\
\ldots +\frac{f^{(k)}(0)}{k!}\cdot (\a_1t+\a_2t^2+\ldots +\a_kt^k)^k 
\text{ modulo } t^{k+1}
\end{multline}
so the linear action of $\varphi$ on the $k$-jet $(f'(0),f''(0)/2!,\ldots,
f^{(k)}(0)/k!)$ is given by the following matrix multiplication:
\begin{equation*}
(f'(0),f''(0)/2!,\ldots,f^{(k)}(0)/k!) \cdot 
\left(\begin{array}{ccccc}
\a_1 & \a_2 & \a_3 & \cdots  & \a_k \\
0        & \a_1^2 & 2\a_1\a_2 & \cdots &  \a_1\a_{k-1}+\ldots +\a_{k-1}\a_1 \\
0        & 0       & \a_1^3  & \cdots & 3\a_1^2\a_{k-2}+\ldots \\
\cdot    & \cdot   & \cdot    & \cdot &  \cdot \\
0 & 0 & 0 & \cdots  & \a_1^k 
\end{array} \right)
\end{equation*}
where the matrix has general entry 
\[(\GG_{k})_{i,j}=\sum_{\substack{s_1,\ldots s_i \in \ZZ_+ \\
s_1+\ldots +s_i=j}}\a_{s_1}\ldots \a_{s_i} \text{ for } 1\le i,j\le k.\] 
$\GG_k$ sits in an exact sequence of groups
$1 \rightarrow U_k \rightarrow \GG_k \rightarrow \CC^* \rightarrow
1$,
 where $\GG_k \to \CC^*$ is the morphism $\varphi \to
\varphi'(0)=\a_1$ in the notation used above, and
\begin{equation}\label{uuk}
\GG_k=U_k \rtimes \CC^*
\end{equation}
is a $\CC^*$-extension of the unipotent group $U_k$. With the above identification, $\CC^*$ is
the subgroup of diagonal matrices satisfying $\a_2=\ldots =\a_k=0$ and
$U_k$ is the unipotent radical of $\GG_k$, consisting of matrices of the form above with $\a_1=1$. The action
of  $\l \in \CC^*$ on $k$-jets is thus described by
\[\l\cdot (f',f'',\ldots ,f^{(k)})=(\l f',\l^2 f'',\ldots,
\l ^kf^{(k)})\]

Following \cite{dem}, we introduce the Green-Griffiths vector bundle $E_{k,m}^{GG}$ whose fibres are complex-valued polynomials $Q(f',f'',\ldots ,f^{(k)})$ on the fibres of $J_kX$ of weighted degree $m$ with respect to the $\CC^*$ action above, that is, they satisfy
\[Q(\l f',\l^2 f'',\ldots, \l^k f^{(k)})=\l^m Q(f',f'',\ldots,
f^{(k)}).\]
The fibrewise $\GG_k$ action on $J_kX$ induces an action on $E_{k,m}^{GG}$. Demailly in \cite{dem} defined the  bundle of invariant jet differentials of order $k$ and weighted degree $m$ as the subbundle $E_{k,m}^n\subset E_{k,m}^{GG}$ of polynomial differential operators $Q(f,f',\ldots, f^{(k)})$ which are invariant under $U_k$, that is for any $\varphi \in  \GG_k$ 
\[Q((f\circ \varphi)',(f\circ \varphi)'', \ldots, (f\circ
\varphi)^{(k)})=\varphi'(0)^m\cdot Q(f',f'',\ldots, f^{(k)}).\]
We call $E_k^n=\oplus_m E_{k,m}^n=(\oplus_mE_{k,m}^{GG})^{U_k}$ the Demailly-Semple bundle of invariant jet differentials
\subsection{Compactification of $J_kX/\GG_k$}\label{subsec:comp}





We recall Demailly's construction from \cite{dem} of a smooth relative compactification of the geometric quotient $J_k^{\mathrm{reg}}X/\GG_k$, where $J_k^{\mathrm{reg}}X \subset J_kX$ is the bundle of regular $k$-jets, that is, $k$-jets such that $f'(0)\neq 0$.   
This smooth compactification is constructed as an iterated tower of projectivized bundles over $X$. Demailly in \cite{dem} uses the term Semple $k$-jet bundle and in this paper we will call this bundle the Demailly-Semple bundle. 

Let $(X,V)$ be a directed manifold of dimension $\dim(X)=n$ and $V \subseteq T_X$ a subbundle of rank $rk(V)=r$. We associate to $(X,V)$  an other directed manifold $(\tilde{X},\tilde{V})$, where $\tilde{X}=\PP(V)$ is the projectivised bundle and $\tilde{V}$ is the subbundle of $T_{\tilde{X}}$ defined
fibrewise using the natural projection $\pi : \tilde{X} \to X$ as follows:
\[\tilde{V}_{(x_0,[v_0])}=\{\xi \in T_{\tilde{X},(x_0,[v_0])}|\pi_*(\xi) \in \CC \cdot v_0\}.\]
for any $x_0 \in  X$ and $v_0 \in T_{X,{x_0}}\setminus \{0\}$. We also have a lifting operator which assigns to a germ of a holomorphic curve $f:(\CC,0) \to X$ tangent to $V$ the germ of the holomorphic curve
$\tilde{f}:(\CC, 0) \to \tilde{X}$ tangent to $\tilde{V}$ defined as $\tilde{f}(t)=(f(t),[f'(t)])$.

Let $X\subset \PP^{n+1}$ be a projective hypersurface. Following Demailly \cite{dem}, we define inductively the $k$-jet bundle $X_k$ and the associated subbundle $V_k \subset T_{X_k}$ by iterating the above construction for $V=T_X$, that is: 
\[(X_0,V_0)=(X,T_X),\text{      } (X_k,V_k)=(\xt_{k-1},\vt_{k-1}).\] 
Therefore, 
\[\dim X_k= n+k(n-1),\ \ \ \ \ \mathrm{rank}V_k=n-1,\]
and the construction can be described inductively by the following exact sequences:
\begin{equation}\label{sempleexact1}
\xymatrix{
    0 \ar[r] & T_{X_k/X_{k-1}} \ar[r] & V_k \ar[r]^-{(\pi_k)_*} & \calo_{X_k(-1)}  \ar[r] & 0}
\end{equation}
\begin{equation*}
\xymatrix{   0 \ar[r] & \calo_{X_k} \ar[r] & \pi_k^*V_{k-1}\otimes \calo_{\PP(V_k)}(-1) 
 \ar[r] & T_{X_k/X_{k-1}}  \ar[r] & 0}
  \end{equation*}
where $\pi_k : X_k \to X_{k-1}$ is the natural projection and $(\pi_k)_*$ is its differential. 
Iterating these we get projections $\pi_{j,k}=\pi_{j+1} \circ \ldots \circ \pi_{k-1} \circ \pi_k: X_k \to X_j$ for $j<k$. 
With this notation $\pi_{0,k}: X_k \to X=X_0$ is a locally trivial holomorphic fibre bundle over $X$, and the fibres $X_{k,x}=\pi_{0,k}^{-1}(x)$ are $k$-stage towers of $\PP^{n-1}$ bundles. 

\begin{theorem}[\cite{dem}]
Suppose that $n>2$. The quotient $J_k^{\reg}X/\GG_k $ has the structure of
a locally trivial bundle over $X$ and there is a holomorphic embedding 
$J_k^{\reg}X/\GG_k\hookrightarrow X_k$ which identifies $J_k^{\reg}X/\GG_k$ with $X_k^{\reg}$, that is the set of points in $X_k$ of the form
$f[k](0$) for some non singular $k$-jet $f$. In other words $X_k$ is a relative compactification of $J_k^{\reg}X/\GG_k$ over X. Moreover, one has the direct image formula:
\[(\pi_{0,k})_*\mathcal{O}_{X_k}(m)=\mathcal{O}(E_{k,m}T_X^*).\]
\end{theorem}

\section{Equivariant cohomology and localisation}\label{sec:equiv}

This section is a brief introduction to equivariant cohomology and localisation. For
more details, we refer the reader to \cite{bgv,bsz}. 

Let $\kt\cong U(1)^n$ be the maximal compact subgroup of
$T\cong(\CC^*)^n$, and denote by $\mathfrak{t}$ the Lie algebra of $\kt$.  
Identifying $T$ with the group $\CC^n$, we obtain a canonical basis of the weights of $T$:
$\lambda_1,\ldots ,\lambda_n\in\mathfrak{t}^*$. 

For a manifold $M$ endowed with the action of $\kt$, one can define a
differential $d_\kt$ on the space $S^\bullet \mathfrak{t}^*\otimes
\Omega^\bullet(M)^\kt$ of polynomial functions on $\mathfrak{t}$ with values
in $\kt$-invariant differential forms by the formula:
\[   
[d_\kt\alpha](X) = d(\alpha(X))-\iota(X_M)[\alpha(X)],
\]
where $X\in\mathfrak{t}$, and $\iota(X_M)$ is contraction by the corresponding
vector field on $M$. A homogeneous polynomial of degree $d$ with
values in $r$-forms is placed in degree $2d+r$, and then $d_\kt$ is an
operator of degree 1.  The cohomology of this complex--the so-called equivariant de Rham complex, denoted by $H^\bullet_T(M)$, is called the $T$-equivariant cohomology of $M$. Elements of $H_T^\bullet (M)$ are therefore polynomial functions $\mathfrak{t} \to \Omega^\bullet(M)^K$ and there is an integration (or push-forward map) $\int: H_T^\bullet(M) \to H_T^\bullet(\mathrm{point})=S^\bullet \mathfrak{t}^*$ defined as  
\[(\int_M \alpha)(X)=\int_M \alpha^{[\mathrm{dim}(M)]}(X) \text{ for all } X\in \mathfrak{t}\]
where $\alpha^{[\mathrm{dim}(M)]}$ is the differential-form-top-degree part of $\alpha$. The following theorem is the Atiyah-Bott-Berline-Vergne localisation theorem in the form of \cite{bgv}, Theorem 7.11. 
\begin{theorem}[Atiyah-Bott \cite{ab}, Berline-Vergne \cite{BV}]\label{abbv} Suppose that $M$ is a compact manifold and $T$ is a complex torus acting smoothly on $M$, and the fixed point set $M^T$ of the $T$-action on M is finite. Then for any cohomology class $\a \in H_T^\bullet(M)$
\[\int_M \alpha=\sum_{f\in M^T}\frac{\a^{[0]}(f)}{\mathrm{Euler}^T(T_fM)}.\]
Here $\mathrm{Euler}^T(T_fM)$ is the $T$-equivariant Euler class of the tangent space $T_fM$, and $\alpha^{[0]}$ is the differential-form-degree-0 part of $\alpha$. 
\end{theorem}

The right hand side in the localisation formula considered in the fraction field of the polynomial ring of $H_T^\bullet (\mathrm{point})=H^\bullet(BT)=S^\bullet \mathfrak{t}^*$ (see more on details in \cite{ab,bgv}). Part of the statement is that the denominators cancel when the sum is simplified. We start with a toy enumerative example to demonstrate how localisation works.

\begin{example}[How many lines intersect 2 given lines and go through a point in $\PP^3$?]\label{example}
We think points, lines and planes in $\PP^3$ as $1,2,3$-dimensional subspaces in $\CC^4$. For $R \in \grass(3,\CC^4), L \in \grass(1,\CC^4)$ define
\[C_2(R)=\{V\in \grass(2,4): V\subset R\},\ C_1(L)=\{V \in \grass(2,4):L \subset V\}\]

Standard Schubert calculus says that $C_1(L)$ (resp $C_2(R)$) represents the cohomology class $c_1(\tau)$ (resp $c_2(\tau)$) where $\tau$ is the tautological rank 2 bundle over $\grass(2,4)$, and the answer can be formulated as  
\[C_1(L_1)\cap C_1(L_2) \cap C_2(R)=\int_{\grass(2,4)}c_1(\tau)^2c_2(\tau).\]
The fixed point data for equivariant localisation is the following.  
\begin{itemize}
\item Let the diagonal torus $T^4 \subset \GL(4)$ act on $\CC^4$ with weights $\mu_1,\mu_2,\mu_3,\mu_4\in \mathfrak{t}^*\subset H_T^*(pt)$.  
\item The induced action on $\grass(2,4)$ has ${4 \choose 2}$ fixed points, namely, the coordinate subspaces indexed by pairs in the set $\{1,2,3,4\}$.
\item The tangent space of $\grass(2,4)$ at the fixed point $(i,j)\in \grass(2,4)^T$  is $(\CC^2)_{i,j}^* \otimes \CC_{s,t}^2$, where $\{s,t\}=\{1,2,3,4\}\setminus \{i,j\}$, and $\CC^2_{i,j}\in \grass(2,4)$ is the subspace spanned by the $i,j$ basis. Therefore, the weights on $T_{(i,j)}\grass$ are $\mu_s-\mu_i,\mu_s-\mu_j$ with $s\neq i,j$. 
\item The weights of $\tau$ are identified with the Chern roots, so $c_i(\tau)$ is represented by the $i$th elementary symmetric polynomial in the weights of $\tau$.  
\end{itemize}
Theorem \ref{abbv} then gives
\begin{equation}\label{ABBV}
\int_{\grass(2,4)}c_1(\tau)^2c_2(\tau)=\sum_{\sigma \in S_4/S_2} \sigma \cdot \frac{(\mu_1+\mu_2)^2\mu_1\mu_2}{(\mu_3-\mu_1)(\mu_4-\mu_1)(\mu_3-\mu_2)(\mu_4-\mu_2)}=1.
\end{equation} 
On the right hand side we sum over all ${4 \choose 2}$ fixed points by taking appropriate permutation of the indices. The sum on the right hand side turns out to be independent of the $\mu_i$'s. 
\end{example}

\section{Equivariant localisation on the Demailly-Semple tower}\label{sec:localisation}


For $x\in X$ a linear $T=(\CC^*)^n$ action on the tangent space $\calv_0=T_{X,x}$ at $x$ induces a linear action on the fibre $\calx_k=X_{k,x}$ of the Demailly bundle over $x$ and on the bundle $\calv_k=V_{k}|_{X_{k,x}}$. This gives a local fibrewise $T$ action on the Demailly-Semple bundle $X_k$, and  we aim to apply the localisation formula Proposition \ref{abbv} on the fibres. The fibre $\calx_k$ is a $k$-stage tower of projective bundles, and to  understand the fixed point data and the weights of the action at the fixed points we use the exact sequences \eqref{sempleexact1}, restricted to the fibre over $x$. Note that \eqref{sempleexact1} restricted to $\calx_k$ is $T$-equivariant.  

For $k=1$ we have $\calx_1=\PP(T_{X,x})$ and we get the Euler sequence
\begin{equation}\label{sempleexact1a}
\xymatrix{
    0 \ar[r] & T_{\calx_1} \ar[r] & \calv_1 \ar[r] & \calo_{\calx_1(-1)}  \ar[r] & 0,}
\end{equation}  
Let $\left\{e_1,\ldots, e_n\right\}$ be an eigenbasis for the $T$-action on $\calv_0=T_{X,x}$ with weights $\l_1,\ldots,$$\l_n$. 
As  \eqref{sempleexact1a} is $T$-equivariant, the weights on $\calv_1|_{[e_j]}$ at the fixed point $[e_j]=[0:\dots 0:1:0:\dots :0]\in \calx_1$  are $\l_j$ and $\l_i-\l_j$ for $i \neq j$. 

Now \eqref{sempleexact1} restricted to the fibre $\calx_k$ gives us:
\begin{displaymath}\label{sempleexact1b}
\xymatrix{
    0 \ar[r] & T_{\calx_k/\calx_{k-1}} \ar[r] & \calv_k \ar[r] & \calo_{\calx_k(-1)}  \ar[r] & 0,}.
\end{displaymath}  
 Locally $\calv_{k}$ is the direct sum of the two bundles on the ends. Fix a point $y \in \calx_k$, and let $\calv_{k-1,\pi_*y}$ denote the fibre of $\calv_{k-1}$ at the point $\pi_*y \in \calx_{k-1}$, where $\pi=\pi_{k,k-1}$. If $y$ is a fixed point of the $T$-action on $\calx_{k}$, then $\pi_*y$ is a fixed point on $\calx_{k-1}$, and therefore $\calv_{k-1,\pi_*y}$ is $T$-invariant, acted on by $T$ with weights $w_1,\ldots ,w_n \in \mathrm{Lin}(\l_1,\ldots, \l_n)$ in the eigenbasis $e_1,\ldots ,e_n$. By definition $\calx_{k}=\PP(\calv_{k-1})$; let $y$ be the fixed line corresponding to the weight $w_j$. The weights on $T_{\calx_{k}/\calx_{k-1},y}=T_{\PP(\calv_{k-1,\pi_*y})}$ at $y$ are $w_i-w_j \text{ for } i\neq j$, and the weight on the tautologiacal bundle $\calo_{\calx_{k}(-1)}$ at $y \in \calx_{k}$ is $w_j$, so the weights on $\calv_{k,y}$ are  
\begin{equation*}
w_i-w_j \text{ for } i=1,\ldots n, i\neq j, \text{ and } w_j.
\end{equation*}
Therefore, a fixed point $y=F_{w_1,\ldots ,w_k}$ is characterised by a sequence $(w_1,\ldots, w_k)$ of weights  $w_i\in \mathrm{Lin}(\l_1,\ldots ,\l_n), i=1, \ldots, k$ where 
\begin{enumerate}
\item
$w_1 \in \cals=\left\{\l_1,\ldots ,\l_n\right\}$ 
\item For $i\ge 2$ $w_i\in \cals(w_1,\ldots ,w_{i-1})=\left\{w_{i-1},w-w_{i-1}:w \in \cals(w_1,\ldots, w_{i-2})\right\}^{\neq 0}$
\end{enumerate}
and $A^{\neq 0}=A\setminus \{0\}$ denotes the set of nonzero elements of $A$. 

Here $\cals(w_1,\ldots, w_{i-1})$ collects the weights of the $T$ action on the fibre $\calv_{i-1,F_{w_1,\ldots, w_{i-1}}}$. For $n=k=3$ we collected the fixed point data in Table \ref{table}.
\begin{table}
\caption{Weights on the Demailly-Semple bundle for $n=k=3$.}
\label{table}
\small
\centering
\begin{tabular}{|c|c|c|}
\hline
$\calv_0$ & $\calv_1$ & $\calv_2$ \\
\hline
 & & $\cals(\l_1,\l_1)=\{\l_1,\l_2-2\l_1,\l_3-2\l_1\}$ \\
\cline{3-3}
 & $\cals(\l_1)= \{\l_1,\l_2-\l_1,\l_3-\l_1\}$ & $\cals(\l_1,\l_2-\l_1)=\{2\l_1-\l_2,\l_2-\l_1,\l_3-\l_2\}$ \\
\cline{3-3}
 & & $\cals(\l_1,\l_3-\l_1)=\{2\l_1-\l_3,\l_2-\l_3,\l_3-\l_1\}$ \\
\cline{2-3}
 & & $\cals(\l_2,\l_1-\l_2)=\{\l_1-\l_2,2\l_2-\l_1,\l_3-\l_1\}$ \\
\cline{3-3}
$\{\l_1,\l_2,\l_3\}$ & $\cals(\l_2)=\{\l_1-\l_2,\l_2,\l_3-\l_2\}$ & $\cals(\l_2,\l_2)=\{\l_1-2\l_2,\l_2,\l_3-2\l_2\}$ \\
\cline{3-3}
 & & $\cals(\l_2,\l_3-\l_2)=\{\l_1-\l_3,2\l_2-\l_3,\l_3-\l_2\}$ \\
\cline{2-3}
 & & $\cals(\l_3,\l_1-\l_3)=\{\l_1-\l_3,\l_2-\l_1,2\l_3-\l_1\}$ \\
\cline{3-3}
 & $\cals(\l_3)=\{\l_1-\l_3,\l_2-\l_3,\l_3\}$ & $\cals(\l_3,\l_2-\l_3)=\{\l_1-\l_2,\l_2-\l_3,2\l_3-\l_2\}$ \\
\cline{3-3}
 & & $\cals(\l_3,\l_3)=\{\l_1-2\l_3,\l_2-2\l_3,\l_3\}$ \\
\hline
\end{tabular}\\
\end{table}
In general, we get by induction the following
\begin{lemma}\label{si}
Let $1\le i \le k$ and $w_j \in \cals(w_1,\ldots, w_{j-1})$ for $1\le j \le i$. Then 
\begin{multline}\nonumber
\mathcal{S}(w_1,\ldots, w_{i})=\{\l_j-w_1-\ldots -w_{i},w_1-w_2-\ldots -w_{i}, \ldots ,w_{i-1}-w_{i},w_{i}:1\le j \le n\}^{\neq 0} \setminus \\ \setminus \left\{-(w_{t}+w_{t+1}+\ldots +w_{i}):2\le t \le i\right\},
\end{multline}
where for $i=1$ we define the subtracted set to be the empty set.  
\end{lemma}
\begin{proof}
For $i=1$ and $w_1=\lambda_r$ for some $1\le r \le n$, the weights are
\[\cals(\l_r)=\{\l_j-\l_r,\l_r:j\neq r\}=\{\lambda_j-w_1,w_1:1\le j \le n\}^{\neq 0}\]
as stated. Assume the Lemma holds for $i-1$, and use (2) above at the description of the weights:
\begin{multline}\nonumber
\cals(w_1,\ldots, w_i)=\left\{w_{i},w-w_{i}:w \in \cals(w_1,\ldots, w_{i-1})\right\}^{\neq 0}=\\ 
= \{\l_j-w_1-\ldots -w_{i},w_1-w_2-\ldots -w_{i}, \ldots ,w_{i-1}-w_{i},w_{i}:1\le j \le n\}^{\neq 0} \setminus \\ \setminus \left\{-(w_{t}+w_{t+1}+\ldots +w_{i}):2\le t \le i\right\} 
\end{multline}
as stated.
\end{proof} 
\begin{remark}\label{important}
Note that there is exactly one element of the set 
\[\{\l_j-w_1-\ldots -w_{i},w_1-w_2-\ldots -w_{i}, \ldots ,w_{i-1}-w_{i},w_{i}:1\le j \le n\}\]
which is equal to zero for every choice of $w_1,\ldots, w_i$. We exclude this element by taking the nonzero part of this in set in Lemma \ref{si}. 
\end{remark}

The fixed point set on the fibre $\calx_{k}$ is then $\mathfrak{F}_k=\{F_{w_1,\ldots, w_k}:w_i\in \cals(w_1,\ldots, w_{i-1}\}$. Proposition \ref{abbv} applied to the fibre of the Demailly-Semple bundle gives
\begin{proposition}\label{atiyahbott} Let $\calx_{k}$ be the fibre of the Demailly-Semple $k$-jet bundle over $X$ at $x \in X$ endowed with the induced $T=(\CC^*)^n$ action from $T_{X,x}$. Then for any $\alpha \in H_T^\bullet(\calx_{k})$  
\begin{equation*}
\int_{\calx_{k}}\a = \sum_{F_{w_1,\ldots ,w_k}\in \mathfrak{F}_k} \frac{\a^{[0]}(F_{w_1,\ldots ,w_k})}{\prod_{j=1}^k \prod_{\substack{w \in \cals(w_1,\ldots ,w_{j-1}) \\ w \neq w_j}}(w-w_j)},
\end{equation*}
where $\mathfrak{F}_k=\{F_{w_1,\ldots, w_k}:w_i\in S_i(w_1,\ldots, w_{i-1})\}$ denotes the set of fixed points on the fibre $\calx_{k}$.
\end{proposition}
\begin{proof}
The equivariant Euler class of the tangent bundle of $\calx_{k}$ at the fixed point $F_{w_1,\ldots, w_k}$ is the product of the weights in the tangent directions:
\[\mathrm{Euler}_T(T_{F_{w_1,\ldots ,w_k}}\calx_{k})=\prod_{j=1}^k \mathrm{Euler}_T(T_{F_{w_1,\ldots ,w_j}}\PP(\calv_{j-1,F_{w_1,\ldots ,w_{j-1}}})\]
and the weights on $\calv_{j-1,F_{w_1,\ldots ,w_{j-1}}}$ are collected in $\cals(w_1,\ldots, w_{j-1})$.  
 \end{proof} 
 
 In particular, we have the following 
 \begin{corollary}\label{loconsemple}
 Let $u_i=c_1(\pi_{i,k}^*\calo_{\calx_i}(1))$, $i=1,\ldots, k$ denote the first Chern classes of the canonical line bundles on $\calx_{k}$, and let $\alpha(u_1,\ldots, u_k)$ be a homogeneous degree $\mathrm{dim} \calx_{k}=k(n-1)$ polynomial. Then $\alpha^{[0]}(F_{w_1,\ldots, w_k})=\alpha(w_1,\ldots, w_k)$ and therefore 
\begin{equation}\nonumber
\int_{\calx_{k}}\a(u_1,\ldots, u_k)=\sum_{F_{w_1,\ldots ,w_k}\in \mathfrak{F}_k} \frac{\alpha(w_1,\ldots, w_k)}{\prod_{j=1}^k \prod_{\substack{w \in \cals(w_1,\ldots ,w_{j-1}) \\ w \neq w_j}}(w-w_j)}.
\end{equation}
 \end{corollary}
  


\subsection{Proof of Theorem \ref{maintechnical}: Transforming the localisation formula into iterated residue}\label{subsec:transform}

In this section we prove Theorem \ref{maintechnical} by transforming the right hand side of the formula in Corollary \ref{loconsemple} into an iterated residue motivated by \cite{bsz}. This will enable us to make effective calculations with the cohomology ring of the Demailly-Semple bundle and to prove positivity of the intersection numbers coming up in the Morse inequalities. 

To describe this formula, we will need the notion of an {\em iterated
  residue} (cf. e.g. \cite{szenes}) at infinity.  Let
$\omega_1,\dots,\omega_N$ be affine linear forms on $\CC^k$; denoting
the coordinates by $z_1,\ldots, z_k$, this means that we can write
$\omega_i=a_i^0+a_i^1z_1+\ldots + a_i^kz_k$. We will use the shorthand
$h(\bz)$ for a function $h(z_1\ldots z_k)$, and $\dbz$ for the
holomorphic $n$-form $dz_1\wedge\dots\wedge dz_k$. Now, let $h(\bz)$
be an entire function, and define the {\em iterated residue at infinity}
as follows:
\begin{equation}
  \label{defresinf}
 \ires \frac{h(\bz)\,\dbz}{\prod_{i=1}^N\omega_i}
  \overset{\mathrm{def}}=\left(\frac1{2\pi i}\right)^k
\int_{|z_1|=R_1}\ldots
\int_{|z_k|=R_k}\frac{h(\bz)\,\dbz}{\prod_{i=1}^N\omega_i},
 \end{equation}
 where $1\ll R_1 \ll \ldots \ll R_k$. The torus $\{|z_m|=R_m;\;m=1 \ldots
 k\}$ is oriented in such a way that $\res_{z_1=\infty}\ldots
 \res_{z_k=\infty}\dbz/(z_1\cdots z_k)=(-1)^k$.
We will also use the following simplified notation: $\sires \overset{\mathrm{def}}=\ires.$

In practice, one way to compute the iterated residue \eqref{defresinf} is the following algorithm: for each $i$, use the expansion
 \begin{equation}
   \label{omegaexp}
 \frac1{\omega_i}=\sum_{j=0}^\infty(-1)^j\frac{(a^{0}_i+a^1_iz_1+\ldots
   +a_{i}^{q(i)-1}z_{q(i)-1})^j}{(a_i^{q(i)}z_{q(i)})^{j+1}},
   \end{equation}
   where $q(i)$ is the largest value of $m$ for which $a_i^m\neq0$,
   then multiply the product of these expressions with $(-1)^kh(z_1\ldots
   z_k)$, and then take the coefficient of $z_1^{-1} \ldots z_k^{-1}$
   in the resulting Laurent series.

The second option to compute iterated residues is working step-by-step. First take the residue with respect to $z_k$ by applying the Residue Theorem on $\CC \cup \{0\}$; the residue at $z_k=\infty$ is minus the sum of the residues at the finite poles $w_j=-1/{a_j^k}(a_j^0+\ldots +a_j^{k-1}z_{k-1})$ for those factors where $a_j^k\neq 0$. In particular, if these poles are pairwise different then 
\begin{equation}\label{computingresidue}
\mathrm{Res}_{z_k=\infty}\frac{h(\bz)\,\dbz}{\prod_{i=1}^N\omega_i}=\sum_{j:a_j^k\neq 0} \frac{-h(z_1,\ldots, z_{k-1}, w_j)}{a_j^k \prod_{i \neq j}w_i(z_1,\ldots, z_{k-1},w_j)}=\sum_{j:a_j^k\neq 0}\frac{-h(\bz)_{[z_k \rightarrow w_j]}}{\tilde{\prod}_{i=1}^N(\omega_i)_{[z_k \rightarrow w_j]}}
\end{equation}\label{computingres}
where $\tilde{\prod}_{\gamma \in \Gamma} \gamma=\prod_{\gamma \in \Gamma,\gamma \neq 0} \gamma$
denotes the product of the nonzero elements of the set and $[z_k \rightarrow w_j]$ means we substitute $w_j$ to $z_k$.
Then we take the next residue with respect to $z_{k-1}$ using the linearity of the residue and the same rule, and we iterate the process.    
\begin{example}
The rational expression $\frac{1}{z_1(z_1-z_2)}$ has two different Laurent expansions, but on $|z_1| \ll |z_2|$ we use $\frac{1}{z_1(z_1-z_2)}=\sum_{i=0}^\infty (-1)^i\frac{z_1^{i-1}}{z_2^{i+1}}$ to get $\res_{\bz=\infty} \frac{1}{z_1-z_2}=1$. An other example is $\res_{\bz=\infty}\frac{1}{(z_1-z_2)(2z_1-z_2)}=\mathrm{coeff}
_{(z_1z_2)^{-1}}\frac{1}{z_2^2}(1+\frac{z_1}{z_2}+\frac{z_1^2}{z_2^2}+\ldots)
(1+\frac{2z_1}{z_2}+\frac{4z_1^2}{z_2^2}+\ldots)=3$.
\end{example}
\begin{example} Let us revisit our toy example Example \ref{example}. Define the differential form 
\[\omega=\frac{-(z_2-z_1)^2(z_1+z_2)^2z_1z_2\dbz}{\prod_{i=1}^4(\mu_i-z_1)\prod_{i=1}^4(\mu_i-z_2)}\]
This is a meromorphic form in $z_2$ on $\PP^1$ with poles at $z_2=\mu_i, 1\le i\le 4$ and $z_2=\infty$. The poles at $\mu_i$ are non-degenerate and therefore applying the Residue Theorem we get 
\[\res_{z_2=\infty}\omega=-\sum_{i=1}^4\underbrace
{\frac{(\mu_i-z_1)^2(\mu_i+z_1)^2\mu_iz_1dz_1}{
\prod_{j=1}^4(\mu_j-z_1)\prod_{j\neq i}(\mu_j-\mu_i)}}_{z_2=\mu_i}=\sum_{i=1}^4
-\frac{(\mu_i-z_1)(\mu_i+z_1)^2\mu_iz_1dz_1}{
\prod_{j\neq i}(\mu_j-z_1)\prod_{j\neq i}(\mu_j-\mu_i)}\]
Repeating the same with $z_1$ we get
\begin{multline}\nonumber
\res_{z_1=\infty}\res_{z_2=\infty}\omega=\sum_{i=1}^4 \sum_{j\neq i}
-\frac{(\mu_i-\mu_j)(\mu_i+\mu_j)^2\mu_i\mu_j}{
\prod_{k\neq i,j}(\mu_k-\mu_j)\prod_{j\neq i}(\mu_j-\mu_i)}=\sum_{i=1}^4 \sum_{j\neq i}
\frac{(\mu_i+\mu_j)^2\mu_i\mu_j}{
\prod_{k\neq i,j}(\mu_k-\mu_j)\prod_{k\neq i,j}(\mu_k-\mu_i)}=\\
=2\sum_{\sigma \in S_4/S_2} \sigma \cdot \frac{(\mu_1+\mu_2)^2\mu_1\mu_2}{(\mu_3-\mu_1)(\mu_4-\mu_1)(\mu_3-\mu_2)(\mu_4-\mu_2)}=2
\int_{\grass(2,4)}c_1(\tau)^2c_2(\tau).
\end{multline}
On the other hand, using the above algorithm by expanding the rational form $\omega$ we get
\[\res_{z_1=\infty}\res_{z_2=\infty}\omega=2,\]
giving us the desired result $1$  for the integral. 
\end{example}
We start with the following iterated residue theorem on projective spaces:
\begin{proposition}
  \label{flagresidue} 
  For a polynomial $P(u)$ on $\CC$ we have 
\begin{equation}
  \label{projres}
\sum_{i=1}^n
\frac{P(\lambda_i)}
{\prod_{j\neq i}(\l_j-\l_i)}=\res_{z=\infty}
\frac{P(z)}{\prod_{j=1}^n(\l_j-z)}dz.
\end{equation}
\end{proposition}
\begin{proof}
  We compute the residue on the right hand side of \eqref{projres} using the Residue
  Theorem on the projective line $\CC\cup\{\infty\}$.  This residue is a contour
  integral, whose value is minus the sum of the $z$-residues of the
  form in \eqref{projres}. These poles are at $z=\lambda_j$,
  $j=1\ldots n$, and after canceling the signs that arise, we obtain the
  left hand side of \eqref{projres}.
 \end{proof} 
Now we prove the following iterated residue formula for the cohomology pairings of $\calx_k$.

\begin{proposition}\label{semple} Let $k\ge 2$ and let $\calx_{k}$ be the fibre of the Demailly-Semple $k$-jet bundle over $X$ at $x \in X$ endowed with the induced $T=(\CC^*)^n$ action from $T_{X,x}$. Let $P(u_1,\ldots, u_k)\in H_T^{k(n-1)}(\calx_k)$ be a homogeneous polynomial of degree $\dim(\calx_k)=k(n-1)$ in $u_i=c_1(\pi_{i,k}^*\calo_{\calx_i}(1))$. Then
\begin{equation*}  
\int_{\calx_{k}}P(\mathbf{u})=\sires \frac{\prod_{2\le t_1 \le t_2 \le k} -(z_{t_1}+z_{t_1+1}+ \ldots +z_{t_2})}{\prod_{1 \le s_1 < s_2 \le k} (z_{s_1}-z_{s_1+1}-\ldots -z_{s_2})}\frac{P(z_1,\ldots, z_k)\dbz}{\prod_{j=1}^k\prod_{i=1}^n(\l_i-z_1-\ldots -z_j)}
\end{equation*}
\end{proposition}

\begin{proof}
We use that $\calx_k \to \calx_{k-1}$ is a $\PP^{n-1}$ bundle over $\calx_{k-1}$ and therefore  
integration on $\calx_{k}$ can be achieved first integrating over the fibre followed by integration over $\calx_{k-1}$. That is, the fixed points on $\calx_{k}$ fiber over the fixed point set on $\calx_{k-1}$ and using Corollary \ref{loconsemple} we get 
\begin{multline}\nonumber
\int_{\calx_{k}}P(\mathbf{u})=\sum_{F_{w_1,\ldots ,w_k}\in \mathfrak{F}_k} \frac{P(w_1,\ldots, w_k)}{\prod_{j=1}^k \prod_{w \in \cals(w_1,\ldots ,w_{j-1}),w \neq w_j}(w-w_j)}=\\
=\sum_{F_{w_1,\ldots ,w_{k-1}}\in \mathfrak{F}_{k-1}} \frac{1}{\prod_{j=1}^{k-1} \prod_{\substack{w \in \cals(w_1,\ldots ,w_{j-1})\\ w \neq w_j}}(w-w_j)}\cdot 
\sum_{w_k \in S(w_1,\ldots, w_{k-1})}\frac{P(w_1,\ldots, w_k)}{\prod_{\substack{w \in \cals(w_1,\ldots ,w_{k-1}) \\ w \neq w_k}}(w-w_k)}
\end{multline}
Recall that the weights of the $T$ action on $\pi^{-1}(F_{w_1,\ldots, w_{k-1}})=\PP(\calv_{k-1,F_{w_1,\ldots, w_{k-1}}})$ are collected in the set $\cals(w_1,\ldots w_{k-1}) \subset \mathrm{Lin}(\l_1,\ldots ,\l_n)$, so by Proposition \ref{flagresidue} the second sum, which is the integral on the fibre $\pi^{-1}(F_{w_1,\ldots, w_{k-1}})\simeq \PP^{n-1}$ can be written as a residue
\[\sum_{w_k \in \cals(w_1,\ldots, w_{k-1})}\frac{P(w_1,\ldots, w_k)}{\prod_{w \in \cals(w_1,\ldots ,w_{k-1}),w \neq w_k}(w-w_k)}=\res_{z_k=\infty} \frac{P(w_1,\ldots, w_{k-1},z_k)\dbz }{\prod_{w \in \cals(w_1,\ldots ,w_{k-1})}(w-z_k)}
\]

Iterating this on the Demailly-Semple tower $\calx_k \to \calx_{k-1} \to \ldots \to \calx_1\to \{x\}$ using Proposition \ref{flagresidue} in every step we get
\[\int_{\calx_{k}}P(\mathbf{u})=\sires \frac{P(z_1,\ldots, z_k)\dbz}{\prod_{j=1}^k \prod_{w \in \cals(z_1,\ldots ,z_{j-1})}(w-z_j)}\]

Using the description of $\cals(z_1,\ldots, z_i)$ from Lemma \ref{si} we get the desired iterated residue formula
\begin{equation}\label{ab}
\int_{\calx_{k}}P(\mathbf{u})= \sires \frac{\prod_{2 \le t_1 \le t_2 \le k}-z_{[t_1\ldots t_2]}P(z_1,\ldots, z_k)\dbz}{\prod_{1 \le s_1 < s_2 \le k} (z_{s_1}-z_{[s_1+1\ldots s_2]})\prod_{j=1}^k\prod_{i=1}^n(\l_i-z_{[1\ldots j]})}.
\end{equation}
Note that the term corresponding to $t_1=t_2=j$ in the numerator cancels out with the term $-z_j$ in the denominator coming from the unique zero element in the set $\{\l_s-z_1-\ldots -z_{j-1},z_1-z_2-\ldots -z_{i}, \ldots ,z_{j-2}-z_{j-1},z_{j-1}:1\le s \le n\}$ in Lemma \ref{si}, see Remark \ref{important}.
\end{proof}

We come to the proof of Theorem \ref{maintechnical}.

\begin{proof}[Theorem \ref{maintechnical}] 

We integrate first along the fibres of the bundle $X_k \to X$, which corresponds to push-forward in cohomology. Since $(\pi_{0,k})_*\pi_{0,k}^*(h)=h$, we can repeat the proof of Proposition \ref{semple} with $P=P(\mathbf{u},h)
$ on the fibres, considering $h$ as a constant at this stage. The iterated residue collects the coefficient of $(z_1\ldots z_k)^{-1}$ which is a constant times $h^n$ because the total degree of the rational expression is $n-k$. Then since $\int_X h^n=d$, integration means the substitution $h^n=d$. 
\end{proof}  

\subsection{An example: Euler characteristic of the jet differentials bundle}
In the rest of the present paper focus on projective hypersurfaces $X \subset \PP^{n+1}$ and we use our iterated residue formula to prove the existence of global sections of some twisted jet differentials bundle. In order to prove this we follow \cite{dmr} and use Morse inequalities to reduce the question to the positivity of some appropriately defined tautological integral over the Demailly-Semple bundle. 

However, Theorem \ref{maintechnical} gives closed formula for other topological invariants of the jet differentials bundle as well, here we give the formula for the Euler characteristic
\[\chi(X,E_kT_X^*)=\sum_{i=0}^n (-1)^i \dim H^i(X,E_kT_X^*)\]
of the invariant jet differentials bundle $E_kT_X^*$. This can be expressed with the Chern character of $E_kT_X^*$ and the Todd class of $X$ as the integral 
\[\chi(X,E_{k})=\int_X [ch(E_{k})\cdot Td(T_X)]_n\]
Localisation on the Demailly-tower then gives the iterated residue formula:
\begin{corollary}
\begin{multline*}
\chi(X,\pi_*\calo_{X_k}(\mathbf{a}))=\\
=(-1)^k \int_{X} \mathrm{Res}_{\mathbf{z}=\infty} \frac{\prod_{2\le t_1 \le t_2 \le k} (z_{t_1}+z_{t_1+1}+ \ldots +z_{t_2})ch(\calo_{X_k}(\mathbf{a}))\cdot Td(T_X)}{\prod_{1 \le s_1 < s_2 \le k} (-z_{s_1}+z_{s_1+1}+\ldots +z_{s_2})
\prod_{j=1}^k (z_1+\ldots +z_j)^n}\cdot  
\prod_{j=1}^ks(1/(z_1+\ldots +z_k)) 
\end{multline*}
where 
\[ch(\calo_{X_k}(\mathbf{a}))=e^{a_1z_1+\ldots +a_kz_k}\]
and 
\[Td(T_X)=1+\frac{1}{2}c_1+\frac{1}{12}(c_1^2+c_2)+\ldots\]
\end{corollary}


\section{Proof of Theorem \ref{mainthmone}}\label{sec:proofone}  

Let $X\subset \PP^{n+1}$ be a smooth hypersurface of degree $\deg X=d$ and let $X_k$ denote the $k$-level Demailly-Semple bundle on $X$. We start with recalling the following classical major result which  
connects jet differentials to the Green-Griffiths-Lang conjecture:

\begin{theorem}[Fundamental vanishing theorem \cite{gg,dem,siu1}]\label{demailly}
Assume that there exist integers $k,m>0$ and ample line bundle
$A\to X$ such that
\[H^0(X_k,\calo_{X_{k}}(m) \otimes \pi^*A^{-1})\simeq
H^0(X,E_{k,m}\cotx \otimes A^{-1})\neq 0\] has non zero sections
$\s_1,\ldots ,\s_N$, and let $Z\subset X_k$ be the base locus of
these sections. Then every entire holomorphic curce $f:\CC \to X$
necessarily satisfies $f_{[k]}(\CC)\subset Z$. In other words, for
every global $\GG_k$-invariant differential equation $P$ vanishing on
an ample divisor, every entire holomorphic curve $f$ must satisfy
the algebraic differential equation $P(f'(t),\ldots, f^{(k)}(t))\equiv 0$.
\end{theorem}

Note, that by Theorem 1. of \cite{div2}, 
\[H^0(X,E_{k,m}\cotx \otimes A^{-1})= 0\]
holds for all $m\ge 1$ if $k<n$, so we can restrict our attention to the range $k\ge n$. 

To control the order of vanishing of these differential forms along
the ample divisor we choose $A$ to be --as in  \cite{dmr} -- a
proper twist of the canonical bundle of $X$. Recall that the
canonical bundle of the smooth, degree $d$ hypersurface $X$ is
\[K_X=\calo_X(d-n-2),\]
which is ample as soon as $d\ge n+3$.
The following theorem summarises the results of \S 3 in \cite{dmr}.

\begin{theorem}[Algebraic degeneracy of entire curves \cite{dmr}]\label{germtoentire}
Assume that $n=k$, and there exist a $\delta=\delta(n) >0$ and $D=D(n,\delta)$ such that
\[H^0(X_{n},\calo_{X_{n}}(m) \otimes \pi^*K_X^{-\delta m})\simeq
H^0(X,E_{d,m}\cotx \otimes K_X^{-\delta m})\neq 0 \]
whenever $\deg(X)>D(n,\delta)$ for some $m \gg 0$. 
Then the Green-Griffiths-Lang conjecture holds whenever 
\[\deg(X) \ge \max(D(n,\delta), \frac{n^2+2n}{\delta}+n+2).\]
\end{theorem}

For $(a_1,\ldots, a_k)\in \ZZ^k$ define the following line bundle on $X_k$:
\begin{equation*}
\calox{k}{\baa}=\pi_{1,k}^*\calox{1}{a_1}\otimes \pi_{2,k}^*\calox{2}{a_2}\otimes \cdots \otimes \calox{k}{a_k}.
\end{equation*}  
\begin{proposition}[\cite{dem} Prop. 6.16, \cite{div2}  Prop. 3.2]
\begin{enumerate}
\item If $a_1\ge 3a_2,\ldots ,a_{k-2}\ge 3a_{k-1}$, and $a_{k-1}\ge 2a_k\ge 0$, then line bundle $\calox{k}{\baa}$ is relatively nef over $X$. 
If, moreover, 
\begin{equation}\label{cond}
a_1\ge 3a_2,\ldots ,a_{k-2}\ge 3a_{k-1}\text{ and }a_{k-1}> 2a_k >0
\end{equation} 
holds, then $\calox{k}{\baa}$ is relatively ample over $X$. 
\item 
Let $\calox{}{1}$ denote the hyperplane divisor on $X$. If \eqref{cond} holds, then $\calox{k}{\baa}\otimes \pi_{0,k}^*\calox{}{l}$ is nef, provided that $l \ge 2|\baa|$, where $|\baa|=a_1+\ldots +a_k$. 
\end{enumerate}
\end{proposition}

Theorem \ref{germtoentire} accompanied with the following theorem gives us Theorem \ref{mainthmone}.

\begin{theorem}\label{main}
Let $X\subset \PP^{n+1}$ be a smooth complex hypersurface with ample canonical bundle, that is $\deg X\ge n+3$. If  $a_i=n^{8(n+1-i)}$,$\d=\frac{1}{2n^{8n}}$ and $d>D(n)=6n^{8n}$
then
\[H^0(X_n,\calo_{X_n}(|\baa|) \otimes \pi^*K_X^{-\d |\baa|})\simeq
H^0(X,E_{n,|\baa|}\cotx \otimes K_X^{-\d |\baa|})\neq 0, \] 
nonzero. 
\end{theorem}

To prove Theorem \ref{main} we use the
algebraic Morse inequalities of Demailly and Trapani and replace the cohomological computations of \cite{dmr}
with the study of the iterated residue.  Let $L\to X$
be a holomorphic line bundle over a compact K\"ahler manifold of
dimension $n$ and $E \to X$ a holomorphic vector bundle of rank $r$.
Demailly in \cite{dem00} proved the following 
\begin{theorem}[Algebraic Morse inequalities \cite{dem00,trap}]\label{morse}
Suppose that $L=F\otimes G^{-1}$ is the difference of the nef line
bundles $F,G$. Then for any nonnegative integer $q\in \ZZ_{\ge 0}$ 
\[\sum_{j=0}^q(-1)^{q-j}h^j(X,L^{\otimes m}\otimes E)\le
r \frac{m^n}{n!}\sum_{j=0}^q(-1)^{q-j}{n \choose j}F^{n-j}\cdot
G^j+o(m^n).\] 
In particular, $q=1$ asserts that $L^{\otimes
m}\otimes E$ has a global section for $m$ large provided 
\[F^n-nF^{n-1}G>0.\]
\end{theorem}

For $d>n+3$ the canonical bundle $K_X \simeq \calox{}{d-n-2}$ is ample, and therefore we have the following expressions for $\calox{k}{\baa}$ and $\calox{k}{\baa}\otimes \pi_{0,k}^*K_X^{-\d \baa}$ as a difference of two nef line bundles over $X$:
\[\calox{k}{\baa}=(\calox{k}{\baa}\otimes \pi_{0,k}^*\calox{}{2|\baa|})\otimes (\pi_{0,k}^*\calox{}{2|\baa|})^{-1}\] 
\[\calox{k}{\baa}\otimes \pi_{0,k}^*K_X^{-\d \baa}=(\calox{k}{\baa}\otimes \pi_{0,k}^*\calox{}{2|\baa|})\otimes (\pi_{0,k}^*\calox{}{2|\baa|}\otimes \pi_{0,k}^*K_X^{\d \baa})^{-1}\]

By the Morse inequalities we need to prove the positivity of the following intersection product:
\begin{multline}\label{inkad}
I(n,k,\baa,\d)=(\calox{k}{\baa}\otimes \pi_{0,k}^*\calox{}{2|\baa|})^{n+k(n-1)}-\\
(n+k(n-1))(\calox{k}{\baa}\otimes \pi_{0,k}^*\calox{}{2|\baa|})^{(k+1)(n-1)}\cdot (\pi_{0,k}^*\calox{}{2|\baa|}\otimes \pi_{0,k}^*K_X^{\d \baa}).
\end{multline}
Let $h=c_1(\calox{}{1})$ denote the first Chern class of the tautological line bundle $\calox{}{1}$, $c_l=c_l(T_X)$ for $l=1,\ldots ,n$, and $u_s=c_1(\calox{s}{1})$ for $s=1,\ldots, s$. Then $c_1(K_X)=-c_1=(d-n-2)h$, and the intersection product we have to estimate becomes
\begin{multline}\label{intnumber}
I_{n,k,\baa,\d}(u_1,\ldots, u_k,\pi_{0,k}^*h)=\\ 
=(a_1u_1+\ldots +a_ku_k+2|\baa|\pi^*h)^{(k+1)(n-1)}
\left(a_1u_1+\ldots +a_ku_k+S_{n,k,\d,d}|\baa|\pi^*h\right) 
\end{multline} 
where $S_{n,k,\d,d}=2-(n+k(n-1))(2+\d(d-n-2))$.

In our forthcoming computations we will use the shorthand notation 
\begin{equation}\label{shorthand}
x_{[i\ldots j]}=x_i+x_{i+1}+\ldots +x_j.
\end{equation} 
for the sum of entries of a constant vector $(x_i,\ldots ,x_j)$.
Applying Theorem \ref{maintechnical} we arrive at  
\begin{corollary}\label{mainprop} With the notation in \eqref{inkad}, \eqref{intnumber} and \eqref{shorthand}  
\begin{equation}\label{maine}
I(n,k,\baa,\d)= \int_{X} \sires \frac{(-1)^k\prod_{2\le t_1 \le t_2 \le k} (z_{[t_1\ldots t_2]})I_{n,k,\baa,\d}(z_1,\ldots, z_k,h)\dbz}{\prod_{1 \le s_1 < s_2 \le k} (-z_{s_1}+z_{[s_1+1\ldots s_2]})\prod_{j=1}^k\prod_{i=1}^n(\l_i+z_{[1\ldots j]})}
\end{equation}
where integration on the right hand side means the substitution $h^n=d$.  
\end{corollary}

\begin{remark}\label{comptools}
We can eliminate the $\lambda_i$'s from \eqref{maine} as follows. The Chern classes of $X$ are expressible with $d,h$ using the following identity:
\[(1+h)^{n+2}=(1+dh)c(X),\]
where $c(X)=c(T_X)$ is the total Chern class of $X$. 
Using the shorthand notations in \eqref{shorthand} we get for $1\le j \le n$
\begin{multline}\nonumber
\prod_{i=1}^n (\l_i+z_{[1\ldots j]})=(z_{[1\ldots j]})^n \prod_{i=1}^n (1+\frac{\l_i}{z_{[1\ldots j]}})=(z_{[1\ldots j]})^nc\left(\frac{1}{z_{[1\ldots j]}}\right)=\\ \nonumber =(z_{[1\ldots j]})^n \frac{\left(1+\frac{h}{z_{[1\ldots j]}}\right)^{n+2}}{1+\frac{dh}{z_{[1\ldots j]}}}=
\frac{(z_{[1\ldots j]}+h)^{n+2}}{(z_{[1\ldots j]}+dh)(z_{[1\ldots j]})} \nonumber
\end{multline}
\end{remark}
and we arrive at the following formula:
\begin{proposition}\label{mainformula} Let $I(n,k,\baa,\d)$ be the intersection number on the Demailly-Semple bundle defined in \eqref{inkad}. Then with the notation \eqref{intnumber} 
\begin{equation}\nonumber
I(n,k,\baa,\d)=
\int_{X} \sires \frac{(-1)^k \prod_{1\le t_1 \le t_2 \le k} (z_{[t_1\ldots t_2]})\prod_{j=1}^k (z_{[1\ldots j]}+dh)I_{n,k,\baa,\delta}(\bz,h)\dbz}{\prod_{1 \le s_1 < s_2 \le k} (-z_{s_1}+z_{[s_1+1 \ldots s_2]}) \prod_{j=1}^k (z_{[1\ldots j]}+h)^{n+2}}.
\end{equation}
\end{proposition}

This formula has the pleasant feature that it expresses the aimed intersection number directly in terms of $n,k,\baa,d,\d$. Indeed, the result of the iterated residue is a polynomial in $n,k,\d$ and $h^n$, and integrating over $X$ simply means a substitution $d=h^n$.
\subsection{Computations with the iterated residue for $n=k$}\label{subsec:compresidue}
From now on we assume that $n=k$, focusing on Theorem \ref{main}.
The iterated residue is formally a contour integral, but as we have explained in Sect.~\ref{subsec:transform}, it simply means an expansion of the rational expression respecting the order $1\ll|z_1|\ll \ldots \ll |z_k|$. 
Using the notation introduced in \eqref{shorthand} we have the following expansions in Proposition \ref{mainformula}:

\begin{enumerate}
\item $\frac{1}{z_{[1\ldots j]}+h}=\frac{1}{z_j}\left(1-\frac{z_{[1\ldots j-1]}+h}{z_j}+\left(\frac{z_{[1\ldots j-1]}+h}{z_j} \right)^2-\ldots \right)$ for $j\ge 1$ where for $j=1$ we define $z_{[1\ldots j-1]}=0$.
\item $\frac{z_{[t_1 \ldots t_2]}}{-z_{t_1}+z_{[t_1+1\ldots t_2]}}=1+\frac{2z_{t_1}}{z_{t_2}}\left(1+\frac{z_{t_1}-z_{[t_1+1\ldots t_2-1]}}{z_{t_2}}-\left(\frac{z_{t_1}-z_{[t_1+1\ldots t_2-1]}}{z_{t_2}}\right)^2+\ldots  \right)$ for $1\le t_1<t_2 \le n$.
\end{enumerate}
For $n=k$ we use the notation $I_{n,\baa,\d,d}$ for the form and by \eqref{intnumber}
\begin{equation*}
I_{n,\baa,\d,d}(\bz,h)=(a_1z_1+\ldots +a_nz_n+2|\baa|h)^{n^2-1}\left(a_1z_1+\ldots +a_nz_n+S_{n,\d}|\baa|h-n^2\d|\baa|dh\right)
\end{equation*}
where 
\[S_{n,\d}=2-2n^2+n^2(n+2)\d.\]
Substituting these into Proposition \ref{mainformula} we get the following:
\begin{multline}\label{main2g}
I(n,\baa,\d,d)=(-1)^n\int_{X} \sires \underbrace{\prod_{j=1}^n\left(1+\frac{z_{[1\ldots j-1]}+dh}{z_j}\right)}_{A^0(\bz)} \cdot \\
\underbrace{\prod_{1\le t_1 < t_2 \le n} \left(1+\frac{2z_{t_1}}{z_{t_2}}\left(1+\frac{z_{t_1}-z_{[t_1+1\ldots t_2-1]}}{z_{t_2}}-\ldots  \right)\right)}_{A^1(\bz)} \cdot
\underbrace{\prod_{j=1}^n\left(1-\frac{z_{[1\ldots j-1]}+h}{z_j}+\ldots \right)^{n+2}}_{A^2(\bz)} \\ 
\underbrace{\frac{(a_1z_1+\ldots +a_nz_n+2|\baa|h)^{n^2-1}\left(a_1z_1+\ldots +a_nz_n+S_{n,\d}|\baa|h-n^2\d|\baa|dh\right)}{(z_1 \ldots z_n)^n}}_{B(\bz)} \dbz 
\end{multline}
Let 
\[A(\bz)=A^0(\bz)A^1(\bz)A^2(\bz)\]
denote the product of the first three rational expressions for short. 

\begin{definition} Fix a basis $\{e_1,\ldots e_n\}$ of $\ZZ^n$. For a lattice point $\bi=(i_1,\ldots,i_n)\in \ZZ^n$ we call 
\[\deff(\bi)=ni_1+(n-1)i_2+\ldots +i_n\]
the defect of $\bi$. The positive lattice semigroup is defined as
\[\Lambda^+=\bigoplus_{i<j}\ZZ^{\ge 0}(e_i-e_j)\oplus \bigoplus_{i=1}^n \ZZ^{\le 0} e_i .\]
The negative lattice points are elements of $\Lambda^-=-\Lambda^+$. Finally, for $\baa,\bbb \in \ZZ^n$ we say that $\baa \ge \bbb$ if there is a $\mathbf{c} \in \Lambda^+$ with $\bbb+\mathbf{c}=\baa$.
\end{definition}
We now prove the following theorem which together with the Morse inequalities gives us Theorem \ref{main}. 
\begin{theorem}\label{maintheoremone}
Let $a_i=n^{8(n+1-i)}$ and $\d=\frac{1}{2n^{8n}}$. Then $I(n,\baa,\d,d)>0$ if  $d>6n^{8n}$.
\end{theorem}
For $\bi \in \ZZ^n, j,k\in \ZZ$ let $A_{\bz^{\bi}d^jh^k}$ denote the coefficient of $\bz^{\bi}d^jh^k$ in $A(\bz)$, and use similar notations for coefficients in $B(\bz)$.
\begin{lemma}\label{observation}
$A_{\bz^\bi(dh)^m}=\coeff_{\bz^{\bi}(dh)^m}A(\bz)=0$ unless $\bi \in \Lambda^+$, for any $m\ge 0$.  
\end{lemma}

\begin{proof} From \eqref{main2g} we see that all monomials appearing after multiplying out is the product of terms of the form $\frac{z_i}{z_j},\frac{h}{z_j}$ and $\frac{dh}{z_j}$ with $1\le i<j \le n$, this implies the result. 
\end{proof}  

Let's step back a bit looking at formula \eqref{main2g}. The residue is by definition the coefficient of $\frac{1}{z_1\ldots z_n}$ in the appropriate Laurent expansion of the big rational expression in $z_1,\ldots, z_n$, $n, d, h$ and $\d$, multiplied by $(-1)^n$. We can therefore omit the $(-1)^n$ factor and simply compute the corresponding coefficient. The result is a polynomial in $n,d,h,\d$, and in fact, a relatively easy argument shows that it is a polynomial in $n,d,\d$ multiplied by $h^n$ 

Indeed, giving degree $1$ to $z_1,\ldots ,z_n,h$ and $0$ to $n,d,\d$, the rational expression in the residue has total degree $0$. Therefore the coefficient of $\frac{1}{z_1\ldots z_n}$ has degree $n$, so it has the form $h^n p(n,d,\d)$ with a polynomial $p$. Since $d$ appears only as a linear factor next to $h$, the degree of $p$ in $d$ is $n$. 

Moreover, $\int_X h^n=d$, so the integration over $X$ is simply a substitution $h^n=d$, resulting the equation 
$I(n,\baa,\d,d)=dp(n,\baa,\d,d)$, 
where 
\[p(n,\baa,\d,d)=p_n(n,\baa,\d)d^n+\ldots +p_1(n,\baa,\d)d+p_0(n,\baa,\d)\]
is a polynomial in $d$ of degree $n$. 
The goal is to show that $p_n$ dominates the rest of the polynomials, that is, to prove the following

\begin{proposition}\label{polynomial} For $a_i=n^{8(n+1-i)}$ and $\d=\frac{1}{2n^{8n}}$,
\[p_n>0 \text{ and } |p_{n-l}|<3n^{8ln}p_n\]
\end{proposition}

Theorem \ref{maintheoremone} is a straightforward consequence of this Proposition, applying the following elementary statement:
\begin{lemma}[Fujiwara bound]\label{estimation}
If $p(d)=p_nd^n+p_{n-1}d^{n-1}+\ldots +p_1d+p_0\in \RR[d]$ satisfies the inequalities 
\[p_n>0;\ \ |p_{n-l}|<D^l |p_n| \text{ for } l=1,\ldots n,\]
then $p(d)>0$ for $d>2D$.  
\end{lemma}

\subsection{Estimation of the leading coefficent}\label{subsec:leadingcoeff}
The next goal is to compute the leading coefficient $p_n(n,\baa,\d)$. For $\bi=(i_1,\ldots, i_n) \in \ZZ^n$ let $\Sigma \bi=i_1+\ldots +i_n$ denote the sum of its coordinates. From \eqref{intnumber}
\begin{equation}\label{an}
p_n=\sum_{\Sigma \bi=0}B_{\bz^{\bi}}A_{\bz^{-\bi-\mathbf{1}}(dh)^n}-
n^2\d|\baa|\sum_{\Sigma \bi=-1}B_{\bz^{\bi}(dh)}A_{\bz^{-\bi-\mathbf{1}}(dh)^{n-1}}
\end{equation}
where $\mathbf{1}=(1,\ldots ,1)$. Note that -- according to Lemma \ref{observation} -- some terms on the r.h.s are $0$, since we have not made any restrictions on the relation of $\bi$ to $\Lambda^+$. 

There is a dominant term on the r.h.s, corresponding to $\bi=(0,\ldots,0)$ in the first sum:
\begin{equation}\label{b0}
B_{\mathbf{0}}=B_{\bz^\mathbf{0}}A_{\frac{(dh)^n}{\bz^{\mathbf{1}}}}=(a_1\cdots  a_n)^n {n^2 \choose n,\ldots ,n}
\end{equation}
Here ${m \choose m_1, \ldots, m_s}=\frac{m!}{m_1!\ldots m_s!}$ denotes the multinomial coefficient equal to the coefficient of $x_1^{m_1}\ldots x_s^{m_s}$ in $(x_1+\ldots +x_s)^m$.
We show that the absolute sum of the remaining terms is less than this dominant term, implying a lower bound for $p_n$ when $\d,\baa$ as in Theorem \ref{maintheoremone}. 
According to the choice $a_i=n^{8(n+1-i)}$ we have for $\Sigma \bi=0$  
\begin{equation}\label{bbz}
B_{\bz^{\bi}}={n^2 \choose i_1+n,\ldots ,i_n+n}a_1^{i_1+n}\ldots a_n^{i_n+n}<a_1^{i_1}\ldots a_n^{i_n}B_{\mathbf{0}}=n^{8\deff(\bi)}B_{\mathbf{0}}
\end{equation}
On the other hand $A_{\bz^{-\bi-\mathbf{1}}(dh)^n}=0$ unless $\deff(\bi)\le 0$ in which case
\begin{equation}\label{upperbound}
A_{\bz^{-\bi-\mathbf{1}}(dh)^n}=\sum_{\bi_1+\bi_2=
-\bi}A^1_{\bi_1}A^2_{\bi_2}<2^{-\deff(\bi)}n^{-3\deff(\bi)}
\end{equation}
holds according to the following two lemmas which will be repeatedly used:
\begin{lemma}\label{decomposition} We have the following estimations: 
\begin{enumerate}
\item  $\sharp\{\bi \in \Lambda^+: \Sigma \bi=0, \deff(\bi)=i\} \le (n-1)^{i}$.
\item Let $\bi \in \Lambda^+$, $\Sigma \bi=0$ be fixed and let $s$ be a positive integer. Then
\[\sharp\{(\bi_1,\ldots,\bi_s)\in (\Lambda^+)^s:\bi_1+\ldots +\bi_s=\bi\}\le s^{\deff(\bi)}.\]
\end{enumerate}
\end{lemma}
\begin{proof}
Let $\bi=\sum_{j=1}^{\deff(\bi)}(e_{l_j}-e_{l_j+1})$ be the unique decomposition of $\bi$ into the sum of positive simple roots. We have $n-1$ positive simple roots which gives the first inequality. For the second part note that each summand can be put into any of the $s$ multiindices $\bi_1,\ldots, \bi_s$ which gives us the second inequality.
\end{proof}

\begin{lemma}\label{lemmaai}
Let $\Sigma \bi=0$. Then $A^1_{\bi},A^2_{\bi}<n^{3\deff(\bi)}$.
\end{lemma}

\begin{proof} Denoting by $\bz^{\bi(t_1,t_2)}$ the monomial we pick from the term corresponding to $t_1,t_2$ we get by definition
\begin{multline}\nonumber
|A^1_{\bi}|=\sum_{\sum_{t_1,t_2}\bi(t_1,t_2)=\bi}\prod_{1\le t_1<t_2 \le n} \coeff_{\bz^{\bi(t_1,t_2)}}\left(1+\frac{2z_{t_1}}{z_{t_2}}
\left(1+\frac{z_{t_1}-z_{[t_1+1-\ldots t_2-1]}}{z_{t_2}}-\ldots  \right)\right)\le\\ 
<\sum_{\sum_{t_1,t_2}\bi(t_1,t_2)=\bi}
\prod_{t_1,t_2}2\cdot \mathrm{comb}(\bi^+(t_1,t_2)),
\end{multline}
where $\bi(t_1,t_2)=\bi^+(t_1,t_2)-m_{t_1,t_2}e_{t_2}$ for unique $\bi^+ \in (\ZZ^{\ge 0})^n$, $m_{t_1,t_2}\in \ZZ^{\ge 0}$, and for $\bj=(j_1,\ldots, j_n)\in (\ZZ^{\ge 0})^{n}$ we define $\mathrm{comb}(\bj)={j_1+\ldots +j_n \choose j_1,\ldots ,j_n}$. Note that $\mathrm{comb}(\bj)$ is a summand in $n^{\Sigma \bj}=(1+\ldots +1)^{\Sigma \bj}$ and therefore $\mathrm{comb}(\bj)<n^{\Sigma \bj}/2$. Hence for $\bi^+(t_1,t_2)\neq 0$ we have
\[\mathrm{comb}(\bi^+(t_1,t_2))<\frac{1}{2}n^{\Sigma \bi^+(t_1,t_2)}\le \frac{1}{2}n^{\deff(\bi(t_1,t_2))} \text{ so } \prod_{t_1,t_2}2 \cdot \mathrm{comb}(\bi^+(t_1,t_2))<n^{\deff(\bi)}.\]
On the other hand Lemma \ref{decomposition} with $s={n \choose 2}$ gives
\[\sum_{\bi=\sum_{t_1,t_2}\bi(t_1,t_2)}1 < {n \choose 2}^{\deff(\bi)},\]
and therefore 
\[|A^1_{\bi}|<n^{\deff(\bi)}{n \choose 2}^{\deff(\bi)}<n^{3\deff(\bi)}.\]
Similarly, if we denote by $\bz^{\bi(j)}h^{-\Sigma \bi(j)}$ the term which we pick from the $j$th term of $A^2$ then 
\begin{multline}\nonumber
|A^2_{\bi}|<\sum_{\sum_{j=1}^n\bi(j)=\bi}\prod_{j} \coeff_{\bz^{\bi(j)}}\left(1-\frac{z_1+\ldots z_{j-1}+h}{z_j}+\left(\frac{z_1+\ldots z_{j-1}+h}{z_j} \right)^2-\ldots  \right)^{n+2}<\\ <\sum_{\sum_{j=1}^n\bi(j)=\bi}\sum_{\bs_1(j)+\ldots +\bs_{n+2}(j)=\bi(j)}\prod_{\substack{1\le j\le n,\\ 1\le m\le n+2}}\mathrm{comb}(\bs_m^+(j)).
\end{multline}
Since $\Sigma \bi=0$, we don't have $h$ in the numerator and therefore $\bs_1(j)=\mathbf{0}$ for any $j$. Lemma \ref{decomposition} gives again
\[\sum_{\sum_{j=1}^n\bi(j)=\bi}\sum_{\bs_1(j)+\ldots +\bs_{n+2}(j)=\bi(j)}1<((n-1)(n+1))^{\deff(\bi)},\]
whereas for $\bs_m^+(j)\neq 0$ we have $\mathrm{comb}(\bs_m^+(j))<\frac{1}{2}n^{\deff(\bs_m(j))}$
as before, giving us
\[|A_\bi^2|<((n-1)n(n+1))^{\deff(\bi)}<n^{3\deff(\bi)},\] 
which proves Lemma \ref{lemmaai}.
\end{proof}  

Substituting inequalities \eqref{bbz} and \eqref{upperbound} into \eqref{an} and using Lemma \ref{decomposition} we get
\begin{multline}\label{firstterm}
\sum_{\substack{\bi \neq 0 \\ \Sigma \bi=0}}B_{\bz^{\bi}}A_{\bz^{-\bi-\mathbf{1}}(dh)^n}<
\sum_{i=1}^{n^2}\sum_{\substack{\bi \neq 0,\Sigma \bi=0,\bi \in \Lambda^+ \\ \deff(\bi)=i}}\left(\frac{2}{n^5}\right)^iB_{\mathbf{0}}
=\sum_{i=1}^{n^2}\left(\frac{2}{n^5}\right)^iB_{\mathbf{0}}\sum_{\substack{\bi \neq 0,\Sigma \bi=0, \bi\in \Lambda^+  \\ \deff(\bi)=i}}1<\\
\sum_{i=1}^{n^2}\left(\frac{2}{n^5}\right)^in^i
B_{\mathbf{0}}<\frac{1}{4}B_{\mathbf{0}}
\end{multline}

We can handle the second sum of the r.h.s in \eqref{an} in a similar fashion. For $\Sigma \bi=-1$, and $e_j=(0,\ldots ,1^j,\ldots ,0)$ the $j$th coordinate vector we have
\begin{equation*}
A_{\bz^{-\bi-\mathbf{1}}(dh)^{n-1}}=\sum_{j_2=1}^n\sum_{j_1\le j_2}\sum_{\substack{\bi_1+\bi_2=-\bi-e_{j_1} \\ \bi_1,\bi_2 \in \Lambda^+}}A^1_{\bi_1}A^2_{\bi_2}
\end{equation*}
holds because we have to sum over all terms coming from $A^0$ in \eqref{main2g}. So applying Lemma \ref{decomposition} and Lemma \ref{lemmaai} again, we get
\begin{equation}\nonumber 
|A_{\bz^{-\bi-\mathbf{1}}(dh)^{n-1}}|<\sum_{j_2=1}^n
\sum_{\substack{j_1\le j_2 \\ -\bi-e_{j_1}\in \Lambda^+}}2^{-\deff(\bi+e_{j_1})}n^{-3\deff(\bi+e_{j_1})}
<
\sum_{\substack{1\le j \le n  \\ -\bi-e_{j}\in \Lambda^+}}2^{-(\deff(\bi)+n+1-j)}n^{1-3(\deff(\bi)+n+1-j)}.
\end{equation}
Then, similarly to \eqref{bbz}, for $\Sigma \bi=-1$  
\begin{equation}\label{bbbz}
B_{\bz^{\bi}(dh)}={n^2-1 \choose i_1+n,\ldots ,i_n+n}a_1^{i_1+n}\ldots a_n^{i_n+n}<n^{8\deff(\bi)}B_{\mathbf{0}}
\end{equation}
and therefore by the first part of Lemma \ref{decomposition}  we get
\begin{eqnarray}\label{secondterm}
|\sum_{\Sigma \bi=-1}B_{\bz^{\bi}(dh)}A_{\bz^{-\bi-\mathbf{1}}(dh)^{n-1}}|<\sum_{\Sigma \bi=-1}\sum_{\substack{1\le j \le n \\ \bi+e_j \in \Lambda^-}} 2^{-(\deff(\bi)+n+1-j)}n^{5\deff(\bi)-3n+3j-2} B_\mathbf{0}<\\ \nonumber
<
\frac{1}{n}\sum_{i=1}^{\infty}\sum_{\substack{\bi \in \Lambda^-  \\ \deff(\bi)=-i, \Sigma \bi=0}}\left(\frac{n^5}{2}\right)^{-i}B_{\mathbf{0}}<\frac{1}{n}\sum_{i=1}^{\infty}\left(\frac{2}{n^4}\right)^{i}B_{\mathbf{0}}<\frac{1}{4n^2}B_{\mathbf{0}}.
\end{eqnarray} 
Since $\d=\frac{1}{2n^{8n}}$ and $a_i=n^{8(n+1-i)}$, we have $\d |\baa|<1$ so substituting  \eqref{firstterm} and \eqref{secondterm} into \eqref{an} we get
\begin{equation}\label{pntob0}
p_n>\frac{1}{2}B_{\mathbf{0}}>0
\end{equation}
proving the first statement of Proposition \ref{polynomial}.

\subsection{Estimation of the coefficients $p_{n-l}(n,\baa,\d)$}\label{subsec:othercoeffs}

In this subsection we study the coefficients $p_{n-l}(n,\baa,\d)=\coeff_{d^{n+1-l}}I(n,\baa,\d,d)$ for $1 \le l \le n$ to prove the second part of Proposition \ref{polynomial}. From \eqref{main2g}
\begin{equation}\label{coeffgeneral}
p_{n-l}(n,\baa,\d)=\sum_{s=0}^l \sum_{\Sigma \bi=-s} B_{\bz^{\bi}h^s}A_{\bz^{-\bi-\mathbf{1}}h^{l-s}(dh)^{n-l}}-n^2\d |\baa|\sum_{s=1}^{l+1} \sum_{\Sigma \bi=-s-1} B_{\bz^{\bi}h^s(dh)}A_{\bz^{-\bi-\mathbf{1}}h^{l-s}(dh)^{n-l-1}}.
\end{equation}

\begin{lemma}\label{lemmaotherterms} Let $a_i=n^{8(n+1-i)}$, $\d=\frac{1}{2n^{8n}}$. The dominant term in \eqref{coeffgeneral} is $B_{\bz^{\bi(l)}h^l}A_{\bz^{-\bi-\mathbf{1}}(dh)^{n-l}}$ where
\begin{equation}\label{defsl}
\bi(l)=(\underbrace{0,\ldots, 0}_{n-l},\underbrace{-1,\ldots, -1}_{l})=-e_{n-l+1}-\ldots-e_n,
\end{equation}
that is, the sum of the other terms in \eqref{coeffgeneral} is smaller than half of this dominant term and hence
\[|p_{n-l}(n,\baa,\delta)|<\frac{3}{2}|B_{\bi(l)h^l}A_{\bz^{-\bi(l)-\mathbf{1}}(dh)^{n-l}}|.\]
\end{lemma}

We devote the rest of this section to the proof of this Lemma. We start with studying terms of the first sum in \eqref{coeffgeneral}. For $\Sigma \bi=-s$ 
\begin{equation}\label{bzihs}
B_{\bz^{\bi}h^s}=\left({n^2 \choose s,i_1+n,\ldots, i_n+n}+\left(\frac{S_{n,\d}}{2}-1\right){n^2-1 \choose s-1,i_1+n,\ldots, i_n+n}\right)(2|\baa|)^s \prod_{t=1}^na_t^{i_t+n}
\end{equation}
and therefore
\begin{equation}\label{bhs}
|B_{\bz^{\bi}h^s}|<2n^2{n^2 \choose s,i_1+n,\ldots, i_n+n}(2|\baa|)^s a_1^{i_1+n} \ldots a_n^{i_n+n}.
\end{equation}
Note that many of the terms in \eqref{coeffgeneral} vanish, because $A_{\bz^{-\bi-\mathbf{1}}h^{l-s}(dh)^{n-l}}=0$ unless $-\bi-\mathbf{1}\in \Lambda^+$.

Using \eqref{bhs} and the closed form for $B_{\mathbf{0}}$ in \eqref{b0} we can estimate from above this dominant term as 
\begin{multline}\label{domterm}
|B_{\bz^{\bi(l)}h^l}\underbrace{A_{\bz^{-\bi(l)-\mathbf{1}}(dh)^{n-l}}}_{=1}|<2n^2{n^2 \choose l,\underbrace{n-1,\ldots,n-1}_{l}, \underbrace{n, \ldots, n}_{n-l}}(2|\baa|)^l a_1^{n} \ldots a_{n-l}^n a_{n-l+1}^{n-1} \ldots a_n^{n-1}\\ 
<2n^2(2|\baa|)^ln^{-4l(l+1)}B_{\mathbf{0}} < n^{8ln}B_{\mathbf{0}}.
\end{multline}
When $\bi \neq \bi(l)$ the right hand side of \eqref{bhs} can be estimated using the trivial inequality between multinomial coefficients:
\begin{equation}\label{bzihs2}
|B_{\bz^{\bi}h^s}|<n^{8\deff(\bi-\bi(l))}(2|\baa|)^{s-l}B_{z^{\bi(l)}h^l}.
\end{equation}
Since $A_{\bz^{-\bi-\mathbf{1}}h^{l-s}(dh)^{n-l}}=0$ if $-\bi-\mathbf{1} \notin \Lambda^+$, for the non vanishing terms $-\bi-\mathbf{1} \in \Lambda^+$ must hold and therefore $\deff(\bi(l)-\bi)=\deff(e_1+\ldots +e_{n-l}-\bi-\mathbf{1})\ge 0$.
On the other hand, by \eqref{main2g} for $\Sigma \bi=-s$   
\begin{equation}
A_{\bz^{-\bi-\mathbf{1}}h^{l-s}(dh)^{n-l}}=\sum_{1\le j_1<\ldots <j_l \le n}\sum_{m_1\le j_1, \ldots ,m_l\le j_l}\sum_{\substack{\bi_1+\bi_2=-\bi -e_{m_1}-\ldots -e_{m_l} \\ \bi_1,\bi_2 \in \Lambda^+}}A^1_{\bz^{\bi_1}}A^2_{\bz^{\bi_2}h^{l-s}}
\end{equation}
where in this summation we pick $\frac{z_{m_i}}{z_{j_i}}$ from the $j_i$th term of $A^0$, and $\frac{dh}{z_s}$ from the $s$th term if $s \notin \{j_1,\ldots, j_l\}$. Note that $\Sigma \bi_1=0$ and $\Sigma \bi_2=s-l$, otherwise the corresponding coefficients are zero. 
\begin{lemma}\label{lemmaaih}
Let $\Sigma \bi=-s$. Then 
\[|A^2_{\bz^\bi h^s}|<n^{3\deff(\bi)+s}.
\]
\end{lemma}
\begin{proof}
The proof is analogous to the proof of Lemma \ref{lemmaai}: we first allocate $s$ factors in the denominator of $\bz^\bi$ and pair all of them with $h$ in the numerator; we can choose these $s$ factors less than $n^s$ different ways. Then repeat the argument in the proof of Lemma \ref{lemmaai}.
\end{proof}  

Applying Lemma \ref{lemmaai} and Lemma \ref{lemmaaih} we get the following upper bound:
\begin{multline}\nonumber
|A_{\bz^{-\bi-\mathbf{1}}h^{l-s}(dh)^{n-l}}|<\sum_{\substack{1\le j_1<\ldots <j_l \le n \\ m_1\le j_1, \ldots ,m_l\le j_l}}\sum_{\substack{\bi_1+\bi_2=-\bi-e_{m_1}-\ldots -e_{m_l}\\ \bi_1,\bi_2 \in \Lambda^+}}n^{-3(\deff(\bi)+(n+1-m_1)+\ldots +(n+1-m_l))+l-s}<\\ \nonumber
\sum_{\substack{1\le m_1<\ldots <m_l \le n\\ \bi_1+\bi_2=-\bi-e_{m_1}-\ldots -e_{m_l}\\ \bi_1,\bi_2 \in \Lambda^+}}(n+1-m_l)(n-m_{l-1})\cdot \ldots \cdot (n-l+2-m_1)n^{-3(\deff(\bi)+(n+1-m_1)+\ldots +(n+1-m_l))+l-s}<\\ \nonumber
<\sum_{\substack{1\le m_1<\ldots <m_l \le n \\ \bi_1+\bi_2=-\bi-e_{m_1}-\ldots -e_{m_l}\\ \bi_1,\bi_2 \in \Lambda^+}} n^{-3\deff(\bi-\bi(l))+l-s}
\end{multline}
where we used the following inequality for $1\le m_1 < \ldots <m_l \le n$:
\[(n+1-m_l)(n-m_{l-1})\ldots (n-l+2-m_1)n^{-(n+1-m_1)-\ldots -(n+1-m_l)}\le n^{-l-(l-1)-\ldots -1}=n^{\deff(\bi(l))}.\]
Applying Lemma \ref{decomposition} again we get 
\begin{multline}\label{azihls}
|A_{\bz^{-\bi-\mathbf{1}}h^{l-s}(dh)^{n-l}}|<\sum_{\substack{1\le m_1<\ldots <m_l \le n \\ \bi+e_{m_1}+\ldots +e_{m_l}\in \Lambda^-}}2^{-\deff(\bi)-(n+1-m_1)-\ldots -(n+1-m_l)}n^{3(\deff(\bi(l)-\bi))+l-s}\le\\
\le \sum_{\substack{1\le m_1<\ldots <m_l \le n \\ \bi+e_{m_1}+\ldots +e_{m_l}\in \Lambda^-}}2^{\deff(\bi(l)-\bi)}n^{3(\deff(\bi(l)-\bi))+l-s}  
\end{multline}
where, again, $(n+1-m_1)+\ldots +n+1-m_l\ge 1+2+\ldots +l=-\deff(\bi(l))$.
Putting \eqref{bzihs2} and \eqref{azihls} together we can estimate the first sum in \eqref{coeffgeneral} as follows:
\begin{multline}\label{nnnn}
\left|\sum_{s=0}^l \sum_{\substack{\Sigma \bi=-s \\ \bi \neq \bi(l)}} B_{\bz^{\bi}h^s}A_{\bz^{-\bi-\mathbf{1}}h^{l-s}(dh)^{n-l}}\right|<
\sum_{s=0}^l \sum_{\substack{1\le m_1<\ldots <m_l \le n \\ \bi+e_{m_1}+\ldots +e_{m_l}\in \Lambda^- \\  \Sigma \bi=-s,\bi \neq \bi(l)}}2^{\deff(\bi(l)-\bi)}n^{-5\deff(\bi(l)-\bi)+l-s}(2|\baa|)^{s-l}B_{\bz^{\bi(l)}h^l}\\
<\sum_{s=0}^l \sum_{\substack{1\le m_1<\ldots <m_l \le n \\ \bi+e_{m_1}+\ldots +e_{m_l}\in \Lambda^- \\ \Sigma \bi=-s,\bi \neq \bi(l)}}n^{-4\deff(\bi(l)-\bi)}n^{(8n-1)(s-l)}B_{\bz^{\bi(l)}h^l}.
\end{multline}
Observe that 
\begin{description}
\item For $\Sigma \bi=-l$ we have 
\begin{multline}\nonumber
\bi+e_{m_1}+\ldots +e_{m_l}\in \Lambda^- \Rightarrow \deff(\bi+e_{m_1}+\ldots +e_{m_l})=\deff(\bi)+(n+1-m_1)+\ldots \\
\ldots+(n+1-m_l)\le 0
\Rightarrow \deff(\bi-\bi(l))\le (m_1+l-n-1)+(m_2+l-n-2)+\ldots +(m_l-n).
\end{multline}
Therefore using the temporary notation $r_i=m_i+l-n-i\le 0$, we get
\begin{multline}\label{nnn}
\sharp \{1\le m_1< \ldots <m_l\le n:\bi+e_{m_1}+\ldots +e_{m_l}\in \Lambda^-\}<\\
<\sharp \{r_1,\ldots ,r_l\le 0:r_1+\ldots +r_l>\deff(\bi-\bi(l)\}<l^{\deff(\bi(l)-\bi)}.
\end{multline}
\item For $\Sigma \bi=-s>-l$, clearly 
\begin{multline}\nonumber
\sharp \{1\le m_1< \ldots <m_l\le n:\bi+e_{m_1}+\ldots +e_{m_l}\in \Lambda^-\}\le\\
<\sharp \{1\le m_1< \ldots <m_l\le n:(\underbrace{\bi-e_n-\ldots-e_{n-l+s+1}}_{\Sigma=-l})+e_{m_1}+\ldots +e_{m_l}\in \Lambda^-\}<\\
l^{\deff(\bi(l)-\bi)+1+\ldots +(l-s)}.
\end{multline}
\end{description}
Substituting these into \eqref{nnnn} we get
\begin{equation}\nonumber
|\sum_{s=0}^l \sum_{\substack{\Sigma \bi=-s\\ \bi \neq \bi(l)}} B_{\bz^{\bi}h^s}A_{\bz^{-\bi-\mathbf{1}}h^{l-s}(dh)^{n-l}}|<\sum_{s=0}^l \sum_{\substack{\Sigma \bi=-s \\ \bi \neq \bi(l)}}l^{\deff(\bi(l)-\bi)+1+\ldots +(l-s)}n^{-4\deff(\bi(l)-\bi)}n^{(8n-1)(s-l)}B_{\bz^{\bi(l)}h^l}<
\end{equation}
\begin{multline}\nonumber
<\sum_{s=0}^l\sum_{m=1}^\infty \sum_{\substack{\Sigma \bi=-s \\\deff(\bi(l)-\bi)=m}}n^{-3m+(7n-1)(s-l)}B_{\bz^{\bi(l)}h^l}<
\sum_{s=0}^l\sum_{m=1}^\infty n^{-2m+(7n-1)(s-l)}B_{\bz^{\bi(l)}h^l}< \\
<\sum_{s=0}^l \frac{1}{8}n^{(7n-1)(s-l)} B_{\bz^{\bi(l)}h^l}<\frac{1}{4}B_{\bz^{\bi(l)}h^l}.
\end{multline}
To summarize our results, since $A_{\bz^{-\bi(l)-\mathbf{1}}(dh)^{n-l}}=1$, we get
\begin{equation}\label{one}
\left|\sum_{s=0}^l \sum_{\Sigma \bi=-s, \bi \neq \bi(l)} B_{\bz^{\bi}h^s}A_{\bz^{-\bi-\mathbf{1}}
h^{l-s}(dh)^{n-l}}\right|<\frac{1}{4}B_{\bz^{\bi(l)}h^l}
A_{\bz^{-\bi(l)-\mathbf{1}}(dh)^{n-l}}.
\end{equation} 
The analogous computation for the second sum in \eqref{coeffgeneral} 
shows that for $\delta=\frac{1}{2n^{8n}}$, $a_i=n^{8(n+1-i)}$ we have
\begin{equation}\label{two}
n^2\d |\baa|\left|\sum_{s=1}^{l+1} \sum_{\Sigma \bi=-s-1} B_{\bz^{\bi}h^s(dh)}A_{\bz^{-\bi-\mathbf{1}}h^{l-s}(dh)^{n-l-1}}\right|<
\frac{1}{4}B_{\bz^{\bi(l)}h^l}A_{\bz^{-\bi(l)-\mathbf{1}}(dh)^{n-l}}.
\end{equation} 

Then \eqref{one} and \eqref{two} gives Lemma \ref{lemmaotherterms}. Combined with  \eqref{domterm} and \eqref{pntob0} gives the desired Proposition \ref{polynomial}:
\[|p_{n-l}|<\frac{3}{2}|B_{\bi(l)h^l}A_{\bz^{-\bi(l)-\mathbf{1}}(dh)^{n-l}}|
<\frac{3}{2}n^{8ln}B_{\mathbf{0}}<
3n^{8ln}|p_{n}|,\]
and Theorem \ref{maintheoremone} is proved. This proves Theorem \ref{main} applying the Morse inequalities. Theorem \ref{germtoentire} and Theorem \ref{main} together give Theorem \ref{mainthmone}.

\end{document}